\documentclass[11pt]{amsart}

\input xy

\xyoption{all}

\usepackage{amssymb,amsbsy,amsthm,amsmath,times}

\newtheorem{thm}{Theorem}[section]
\newtheorem{lem}[thm]{Lemma}
\newtheorem{prop}[thm]{Proposition}
\newtheorem{cor}[thm]{Corollary}

\newtheorem{dfn}[thm]{Definition}

\theoremstyle{remark}
\newtheorem*{rmk}{Remark}

\newcommand{\bs}[1]{\boldsymbol{#1}}
\renewcommand{\bf}[1]{\mathbf{#1}}
\renewcommand{\rm}[1]{\mathrm{#1}}
\newcommand{\cal}[1]{\mathcal{#1}}

\newcommand{\bbN}{\mathbb{N}}

\newcommand{\bbZ}{\mathbb{Z}}

\newcommand{\bfX}{\mathbf{X}}

\newcommand{\bfZ}{\mathbf{Z}}

\newcommand{\sfE}{\mathsf{E}}

\renewcommand{\d}{\mathrm{d}}

\newcommand{\G}{\Gamma}

\renewcommand{\S}{\Sigma}

\renewcommand{\a}{\alpha}

\newcommand{\eps}{\varepsilon}

\renewcommand{\l}{\Lambda}
\renewcommand{\o}{\omega}
\newcommand{\s}{\sigma}

\newcommand{\nil}{\mathrm{nil}}

\newcommand{\fin}{\nolinebreak\hspace{\stretch{1}}$\lhd$}

\newcommand{\actson}{\curvearrowright}

\begin{document}

\title[Ergodic-theoretic density-increment]{Ergodic-theoretic implementations of the Roth density-increment argument}
\author{Tim Austin}

\address{Courant Institute\\ New York University\\ New York, NY 10012, USA}

\email{tim@cims.nyu.edu}

\urladdr{http://www.cims.nyu.edu/~tim}

\thanks{Work supported by fellowships from Microsoft Corporation and from the Clay Mathematics Institute}



\date{}

\begin{abstract}
We exhibit proofs of Furstenberg's Multiple Recurrence Theorem and of a special case of Furstenberg and Katznelson's multidimensional version of this theorem, using an analog of the density-increment argument of Roth and Gowers.  The second of these results requires also an analog of some recent finitary work by Shkredov.

Many proofs of these multiple recurrence theorems are already known.  However, the approach of this paper sheds some further light on the well-known heuristic correspondence between the ergodic-theoretic and combinatorial aspects of multiple recurrence and Szemer\'edi's Theorem.  Focusing on the density-increment strategy highlights several close points of connection between these settings.
\end{abstract}

\maketitle

\tableofcontents

\section{Introduction}

In 1975 Szemer\'edi published the first proof of a long-standing conjecture of Erd\H{o}s and Tur\'an concerning arithmetic progressions in dense arithmetic sets.

\begin{thm}[Szemer\'edi's Theorem]\label{thm:Szem}
If $E \subset \bbZ$ admits some $\delta > 0$ for which there are arbitrarily long intervals $[M,N]$ with
\[|E\cap [M,N]| \geq \delta (N-M)\]
(that is, $E$ has `upper Banach density' equal to at least $\delta$), then $E$ also contains for every $k\geq 1$ a nondegenerate arithmetic progression of length $k$:
\[E \supset \{a,a+n,a+2n,\ldots,a + (k-1)n\}\quad\quad\hbox{for some }a\in\bbZ,n\geq 1.\]
\end{thm}

Separate proofs for various special cases were given earlier by Roth and by Szemer\'edi himself.  The thirty-five years subsequent to Szemer\'edi's breakthrough have seen the emergence of a host of alternative approaches to this theorem and several generalizations.

The many techniques that have been brought to bear in this investigation are loosely drawn from three areas of mathematics:
\begin{itemize}
\item graph and hypergraph theory (in work of Szemer\'edi, Solymosi, Nagle, R\"odl, Schacht, Skokan, Gowers and others),
\item ergodic theory (largely building on ideas of Furstenberg and Katznelson),
\item harmonic analysis (following Roth, Bourgain, Gowers, Green, Tao and Shkredov).
\end{itemize}
The alternative arguments constructed from these three bodies of theory sometimes correspond much more closely than is initially apparent, owing to many differences in technical detail that turn out to be quite superficial. No really comprehensive overview of the relations among these approaches is yet available, but fragments of the picture can be found in the papers~\cite{Kra07(surv),GreTao--nildist,Tao06} and in Chapters 10 and 11 of Tao and Vu's book~\cite{TaoVu06}.

The purpose of the present note is to extract one aspect of the harmonic analytic approach --- the `density-increment argument', originating in the early work of Roth~\cite{Rot53} --- and present a natural analog of it in the rather different setting of ergodic theory.  No new theorems will be proved except for some technical results needed on route, but I hope that this alternative presentation of existing ideas will contribute to enhancing the toolkits of those working on this class of problems, and also shed some light on the open questions that remain concerning the density-increment approach.

\subsection{Ergodic Ramsey Theory}

Two years after Szemer\'edi's proof of Theorem~\ref{thm:Szem} appeared, Furstenberg offered in~\cite{Fur77} a very different approach to the same result based on a conversion to a problem in ergodic theory, using what is now referred to as `Furstenberg's correspondence principle'.

A precise formulation of the general correspondence principle can be found, for example, in Bergelson~\cite{Ber96}.  Here we simply recall that Furstenberg proved the equivalence of Szemer\'edi's Theorem to the following:

\begin{thm}[Multiple Recurrence Theorem]\label{thm:MRT}
If $T:\bbZ\actson (X,\mu)$ is a probability-preserving action on a standard Borel probability space and $A \subset X$ is measurable and has $\mu(A) > 0$, then also
\[\liminf_{N\to\infty}\frac{1}{N}\sum_{n=1}^N\mu(A\cap T^{-n}A\cap\cdots\cap T^{-(k-1)n}A) > 0\quad\quad\forall k\geq 1.\]
\end{thm}

Furstenberg's proof of Theorem~\ref{thm:MRT} relied on a powerful structural classification of probability-preserving dynamical systems developed independently by Furstenberg and by Zimmer~(\cite{Zim76.1,Zim76.2}).

Shortly after that proof appeared, Furstenberg and Katznelson realized that only a modest adaptation yields a significantly stronger result.

\begin{thm}[Multidimensional Multiple Recurrence Theorem]\label{thm:multiMRT}
If $T_1, T_2, \ldots, T_d: \bbZ\actson (X,\mu)$ are commuting probability-preserving actions on a standard Borel probability space and $A \subset X$ has $\mu(A) > 0$ then also
\[\liminf_{N\to\infty}\frac{1}{N}\sum_{n=1}^N\mu(A\cap T_1^{-n}A\cap\cdots\cap T_d^{-n}A) > 0.\]
\end{thm}

This appeared in~\cite{FurKat78}.  Theorem~\ref{thm:MRT} follows from Theorem~\ref{thm:multiMRT} by setting $d := k-1$ and $T_i := T^i$ for $i\leq k-1$.  On the other hand, Theorem~\ref{thm:multiMRT} also has a combinatorial consequence that strengthens Szemer\'edi's Theorem:

\begin{thm}\label{thm:multiSzem}
If $E \subset \bbZ^d$ admits some $\delta > 0$ for which there are cuboids $\prod_{i\leq d}[M_i,N_i]$ with $\min_{i\leq d}|N_i - M_i|$ arbitrarily large and
\[\Big|E\cap \prod_{i\leq d}[M_i,N_i]\Big| \geq \delta \prod_{i\leq d}(N_i - M_i),\]
then $E$ also contains the set of vertices of a nondegenerate upright right-angled isosceles simplex:
\[E \supset \{\bf{a},\bf{a}+n\bf{e}_1,\ldots,\bf{a} + n\bf{e}_d\}\quad\quad\hbox{for some }\bf{a}\in\bbZ^d,n\geq 1.\]
\end{thm}

Interestingly, this result went unproven by purely combinatorial means until the development of hypergraph analogs of Szemer\'edi's famous Regularity Lemma by Nagle, R\"odl and Schacht~\cite{NagRodSch06}, Gowers~\cite{Gow07} and~\cite{Tao06(hyperreg)}, more than twenty years later.  In addition, several other purely combinatorial assertions have now been accessed though `Ergodic Ramsey Theory', the subject that emerged from Furstenberg and Katznelson's early developments, including a density version of the Hales-Jewett Theorem~\cite{FurKat91} and a density Ramsey Theorem for subtrees of trees~\cite{FurWei03}.

Within ergodic theory, a great deal of energy has now been spent on obtaining the most precise possible understanding of the averages whose limit infima are the subjects of Theorems~\ref{thm:MRT} and~\ref{thm:multiMRT}; we will return to some of these developments later.

\subsection{The density-increment argument}

The `density-increment argument' was first used by Roth for his early proof of the case $k=3$ of Theorem~\ref{thm:Szem}.  Much more recently, Gowers developed in~\cite{Gow98,Gow01} an extremely sophisticated extension of Roth's approach, and using this was able to give a density-increment proof of the full Szemer\'edi Theorem.

We will not spend time here on the many technical accomplishments involved in Gowers' work, requiring a call to tools from yet other parts of arithmetic combinatorics such as Freiman's Theorem.  Rather we record just a simple statement of the density-increment proposition that lies at its heart.

\begin{prop}\label{prop:dens-inc}
Suppose that $\delta > 0$, that $N$ is sufficiently large and that $E \subset \{1,2,\ldots,N\}$ has $|E| \geq \delta N$ but contains no $k$-term arithmetic progression.  Then there is an arithmetic progression $P \subset \{1,2,\ldots,N\}$ of size at least $N^{((\delta/2)^{k2^k})^{2^{2^{k+8}}}}$ such that \[|E\cap P| \geq (\delta + ((\delta/2)^{k2^k})^{2^{2^{k+8}}})|P|.\] \qed
\end{prop}

This proposition is implicit in~\cite{Gow01}, but does not appear in the above form because Gowers presents his argument in terms of the crucial auxiliary notion of `higher-degree uniformity', and splits the above result into several pieces that are connected via this auxiliary notion.

This kind of uniformity is defined in terms of the Gowers uniformity norms (Section 3 of~\cite{Gow01}; see also Chapter 11 of Tao and Vu~\cite{TaoVu06}) that have since become widely used in additive combinatorics.  Uniformity of degree $1$ can be described simply in terms of the presence of some large values among the Fourier coefficients of $1_E$, regarded as a function on the group $\bbZ/N\bbZ$; this is essentially the notion that Roth uses in his approach for $k=3$.  Higher-degree uniformity extends this property, although it is not so easily described using Fourier analysis.  In his more general setting, Gowers proves on the one hand that if $E$ is sufficiently uniform of degree $k - 2$ then it contains a $k$-term arithmetic progression (Corollary 3.6 in~\cite{Gow01}), and on the other that if $E$ is not sufficiently uniform of degree $k-2$ then we may partition $\{1,2,\ldots,N\}$ into long arithmetic subprogressions such that $E$ has a relative density inside some of these subprogressions that is substantially larger than $\delta$ (Theorem 18.1 in~\cite{Gow01}). This fact can then be used to pick out one such subprogression satisfying the above conclusion (Lemma 5.15 in~\cite{Gow01}).  Proposition~\ref{prop:dens-inc} amounts to the conjunction of these facts.

Our proof of Theorem~\ref{thm:MRT} below takes a similar form (although we should stress that our task is much simpler than Gowers'), using an ergodic-theoretic analog of the notion of `uniformity' arising in work of Host and Kra. Similarly to the presentation in~\cite{Gow01}, we will find that handling the consequences of non-uniformity is the more complicated of the two steps involved.

From Proposition~\ref{prop:dens-inc} a proof of Szemer\'edi's Theorem follows quickly by contradiction.  If $E$ is a counterexample of density $\delta$ and $N$ is sufficiently large, then for a subprogression $P$ as given by Proposition~\ref{prop:dens-inc} we see that $E\cap P$, identified with a subset of $\{1,2,\ldots,|P|\}$ by the obvious affine map, is another counterexample with density that exceeds $\delta$ by an amount depending only on $\delta$ and $k$.  It is contained in a discrete interval of length $|P| = N^{\kappa(\delta,k)}$ for some small fixed $\kappa(\delta,k) > 0$.  Therefore, provided $N$ was sufficiently large to begin with, iterating this construction must eventually turn a counterexample of density at least $\delta$ into a counterexample of density greater than $1$: an obvious contradiction.

In addition to its aesthetic value, Gowers' new proof of
Szemer\'edi's Theorem gives much the best known
bound on how large $N$ must be taken in order that a $k$-term arithmetic progression is certain to be
found in a density-$\delta$ subset $E\subset \{1,2, \ldots,N\}$.  In view of this, it was natural to ask
whether this approach could also be brought to bear on the
multidimensional Theorem~\ref{thm:multiSzem} in order to give a
similarly striking improvement to the bounds available there.  Gowers
poses this problem explicitly and offers some discussion of it in
his survey~\cite{Gow00}.  Recently Shkredov has made the first
serious progress on this problem by essentially solving the case $d
= 2$ in~\cite{Shk06}, applying some important new technical
ideas that are needed to prove and then use a relative of
Proposition~\ref{prop:dens-inc}.  However, a further enhancement of these
ideas that will yield a density-increment proof of the full
Theorem~\ref{thm:multiSzem}, with or without improved bounds,
still seems relatively distant.

\subsection{Outline of this note}

The centrepieces of this note are `density-increment' proofs of the Multiple Recurrence Theorem~\ref{thm:MRT} and the case $d = 2$ of Theorem~\ref{thm:multiMRT}, corresponding to Gowers' and Shkredov's combinatorial implementations of the density-increment argument respectively.

The main steps taken by Gowers and Shkredov do have counterparts in these proofs, but we need different structural results from within ergodic theory to enable them.  These will largely be drawn from recent studies of the `nonconventional ergodic averages' whose
limit infima appear in Theorems~\ref{thm:MRT}
and~\ref{thm:multiMRT}.  In particular we rely on the method of
`characteristic factors', which has emerged through the works of
several researchers since Furstenberg's original paper~\cite{Fur77},
and especially on some of the technical steps in Host and Kra's
proof~(\cite{HosKra05}) of convergence for the averages of~\ref{thm:MRT} and in the
work of Conze and Lesigne~\cite{ConLes84,ConLes88.1,ConLes88.2} and the subsequent
works~\cite{Aus--nonconv,Aus--newmultiSzem} on the multi-dimensional case.  Many other researchers have contributed to this story within ergodic theory, including Rudolph, Zhang, Katznelson, Weiss, Ziegler, Frantzikinakis and Chu, and the reader is referred to~\cite{Aus--thesis} for a more complete discussion.

The basic ergodic theoretic version of the density-increment
argument for Theorem~\ref{thm:MRT} will be introduced in Subsection~\ref{subs:intro-dens-inc} and then used to complete the proof of that theorem later in Section~\ref{sec:1D}.  Although a density increment is central to Shkredov's proof as well, he uses it in a slightly more complicated way, and so in Section~\ref{sec:2D} we introduce the ergodic theoretic analog of this separately and then use it to prove the case $d=2$ of Theorem~\ref{thm:multiMRT}.

On the one hand, I hope that these proofs shed some light on the nature of the density-increment argument.  On the other, it seems that recent progress in ergodic theory is beginning to address some of the problems of extending this approach to give a density-increment proof of the whole of Theorem~\ref{thm:multiMRT} (and so, one might hope, also to give a finitary density-increment proof of Theorem~\ref{thm:multiSzem}, as requested by Gowers).  In the final Section~\ref{sec:2Dprobs} we will draw on results from~\cite{Aus--thesis,Aus--lindeppleasant1,Aus--lindeppleasant2} to sketch some of the further developments suggested by this progress.

\section{Powers of a single transformation}\label{sec:1D}

\subsection{Preliminary discussion}

In this section we show how the density-increment strategy can be used to give a proof of Theorem~\ref{thm:MRT}, building on two important ergodic-theoretic ingredients.  Let us first recall a convenient definition.

\begin{dfn}[Process]\label{dfn:process}
We will refer to a probability-preserving $\bbZ$-system $(X,\mu,T)$ together with a distinguished subset $A$ as a \textbf{process} and denote it by $(X\supset A,\mu,T)$.
\end{dfn}

\begin{dfn}\label{dfn:no-APs-1}
An ergodic process $(X\supset A,\mu,T)$ has \textbf{no $k$-APs in its return times} if
\[\mu(A\cap T^{-n}A\cap \cdots\cap T^{-(k-1)n}A) = 0\quad \hbox{for all}\ n\in \bbZ\setminus\{0\}.\]
\end{dfn}

Clearly if $\mu(A) > 0$ then the above property is stronger than being a counterexample to the Multiple Recurrence Theorem, which requires that the relevant intersections have positive measure on average, not just for a single nonzero $n$.  Since that theorem turns out to be true, the above definition is essentially vacuous, but it will be a convenient handle at various points during the proofs that follow.

The first ingredient we need is a corollary of the recent result of Host and Kra~\cite{HosKra05} that the limiting values of the multiple recurrence averages are precisely controlled by certain special nilrotation factors of a system $(X,\mu,T)$.

\begin{dfn}[Nilrotations]
For any $k\geq 1$ a \textbf{$k$-step nilrotation} is a $\bbZ$-system on a homogeneous space $G/\G$ for $G$ a $k$-step nilpotent Lie group and $\G \leq G$ a cocompact discrete subgroup, where $G/\G$ is endowed with its normalized Haar measure $m$ and the transformation is given by
\[R_g:h\G\mapsto gh\G\]
for some $g \in G$.
\end{dfn}

\begin{thm}[Host-Kra Theorem]\label{thm:HK}
For each $k\geq 2$, any ergodic $\bbZ$-system $\bf{X} = (X,\mu,T)$ has a
factor map $\pi_{k-2}:\bf{X}\to \bf{Z}_{k-2}$ onto a system generated by an inverse sequence of $(k-2)$-step nilrotations such that
\begin{multline*}
\frac{1}{N}\sum_{n=1}^N\int_X f_0\cdot (f_1\circ T^n)\cdot \cdots \cdot (f_{k-1}\circ T^{(k-1)n})\,\d\mu\\
\sim \frac{1}{N}\sum_{n=1}^N\int_X\sfE_\mu(f_0\,|\,\pi_{k-2})\cdot
(\sfE_\mu(f_1\,|\,\pi_{k-2})\circ T^n)\cdot \cdots
\cdot (\sfE_\mu(f_{k-1}\,|\,\pi_{k-2})\circ T^{kn})\,\d\mu
\end{multline*}
as $N\to\infty$ for any $f_0,f_1,\ldots,f_{k-1} \in L^\infty(\mu)$, where the notation asserts that the difference between these two sequences of averages tends to $0$ as $N\to\infty$. \qed
\end{thm}

\begin{rmk}
The above result is often expressed by asserting that the factor $\pi_{k-2}$ is \textbf{characteristic} for the averages in question.  This theorem first appears in~\cite{HosKra05}, where its proof invokes a family of seminorms on $L^\infty(\mu)$ that Host and Kra introduce for this purpose and that are closely analogous to Gowers' uniformity seminorms from~\cite{Gow01}, so offering another point of proximity between the ergodic theoretic and quantitative approaches.  Another proof of Theorem~\ref{thm:HK} has now been given by Ziegler in~\cite{Zie07}, who also shows that the maximal factor of $\bfX$ generated by $(k-2)$-step nilrotations is also the unique minimal factor that is characteristic in the above sense. \fin
\end{rmk}

With the Host-Kra Theorem in mind, the second result that we use is simply the fact that multiple recurrence does hold for nilrotations.

\begin{thm}[Multiple recurrence for nilrotations]\label{thm:nilsys-MRT}
If $R_g \actson G/\G$ is an ergodic nilrotation, $A \subset G/\G$ has positive measure and $K\geq 1$ then there is some $r\geq 1$ such that
\[m(g^{-Kr}A\cap g^{-(K-1)r}A\cap \cdots \cap g^{Kr}A) > 0.\]
\qed
\end{thm}

In fact, this result is considerably simpler than Theorem~\ref{thm:HK}, which really does the heavy lifting in what follows.  The point is that the orbit of the diagonal $\Delta:= \{(x,x,\ldots,x):\ x \in G/\G\}\subset (G/\G)^{2K+1}$ (or rather, its normalized surface measure $m_\Delta$) under the off-diagonal transformation $R_{(g^{-K},g^{-K+1},\ldots,g^K)}$ (which is clearly still a nilrotation acting on $(G/\G)^{2K+1}$) can be shown to equidstribute in some finite union of closed connected nilsubmanifolds of $(G/\G)^{2K+1}$ that contains the whole of this diagonal set.  This follows from strong results classifying all ergodic invariant measures for nilrotations. From this point a fairly elementary argument gives the positivity of
\begin{multline*}
\liminf_{N\to\infty}\frac{1}{N}\sum_{n=1}^N m(g^{-Kn}A\cap g^{-(K-1)n}A\cap \cdots \cap g^{Kn}A)\\ = \liminf_{N\to\infty}\frac{1}{N}\sum_{n=1}^N \int 1_{A\times A\times \cdots\times A}\circ R_{(g^{-K},g^{-K+1},\ldots,g^K)}^n\,\d m_\Delta,
\end{multline*}
and also the fact that these averages actually converge, so this limit infimum is really a limit. A related instance of this argument can be found presented in detail in Section 2 of the work~\cite{BerLeiLes08} by Bergelson, Leibman and Lesigne, who use it for the related end of proving multiple recurrence along certain families of polynomials.  Equidistribution results for nilrotations on which this reasoning can be founded are available in either Ziegler~\cite{Zie05} or Bergelson, Host and Kra~\cite{BerHosKra05}, which in turn build on older works of Parry~\cite{Par69,Par70,Par73}, Lesigne~\cite{Les91} and Leibman~\cite{Lei98(2),Lei05}.

With Theorems~\ref{thm:HK} and~\ref{thm:nilsys-MRT} at our disposal, it is relatively easy to lay out a density-increment proof of the full Multiple Recurrence Theorem.  However, it is important to observe right away that this is a rather perverse thing to do, because the above two ingredients also imply that theorem through the following even quicker argument:
\begin{itemize}
\item given our process $(X \supset A,\mu,T)$, we wish to prove that
\[\liminf_{N\to\infty}\frac{1}{N}\sum_{n=1}^N\mu(A\cap T^{-n}A\cap \cdots \cap T^{-(k-1)n}A) > 0,\]
so by Theorem~\ref{thm:HK} it suffices to prove instead that
\[\liminf_{N\to\infty}\frac{1}{N}\sum_{n=1}^N\int_X \prod_{i=0}^{k-1}(\sfE(A\,|\,\pi_{k-2})\circ T^{in})\,\d\mu > 0\]
with $\pi_{k-2}:\bfX\to \bfZ_{k-2}$ the inverse limit of nilrotation factors from that theorem (and where we write $\sfE(A\,|\,\pi_{k-2})$ as short for $\sfE(1_A\,|\,\pi_{k-2})$);

\item this, in turn, will follow if we prove that
\[\liminf_{N\to\infty}\frac{1}{N}\sum_{n=1}^N\mu(B\cap T^{-n}B\cap \cdots \cap T^{-(k-1)n}B) > 0\]
where $B := \{\sfE(A\,|\,\pi_{k-2}) > \eps\}$ for any positive $\eps$ chosen so small that $\mu(B) > 0$ (for example, $\eps \leq \mu(A)/2$ will do);

\item finally, importing a simple trick from~\cite{FurKat78}, this follows by choosing a further factor $\a:Z_{k-2}\to G/\G$ onto a finite-dimensional nilrotation such that \[\|\sfE(A\,|\pi_{k-2}) - \sfE(A\,|\,\a\circ\pi_{k-2})\|_1 < \frac{\eps}{100(k+1)}\]
(which is possible because $\pi_{k-2}$ is generated by an inverse sequence of such further factors $\a$) and letting
\[C := \Big\{\sfE(B\,|\,\a\circ\pi_{k-2}) > 1 - \frac{1}{k+1}\Big\},\]
for which we now easily deduced that
\begin{multline*}
\liminf_{N\to\infty}\frac{1}{N}\sum_{n=1}^N\mu(B\cap T^{-n}B\cap \cdots \cap T^{-(k-1)n}B)\\ \geq \frac{1}{2}\liminf_{N\to\infty}\frac{1}{N}\sum_{n=1}^N\mu(C\cap T^{-n}C\cap \cdots \cap T^{-(k-1)n}C),
\end{multline*}
which Theorem~\ref{thm:nilsys-MRT} shows is strictly positive.
\end{itemize}
(This proof is also essentially that used in~\cite{BerLeiLes08} for their instance of polynomial recurrence.)

Therefore the point of this section is not to provide a serious new approach to the Multiple Recurrence Theorem, but rather to exhibit the density-increment strategy in a setting familiar to ergodic theorists.

The reason why the approach to multiple recurrence just sketched does not have a clear analog among quantitative proofs of Szemer\'edi's Theorem is hidden in our appeal to Theorem~\ref{thm:HK}.  In fact, the technical result that drives Gowers' work is really more analogous to the following easy corollary of Theorem~\ref{thm:HK} than to Theorem~\ref{thm:HK} itself:

\begin{cor}\label{cor:corn-on-nilsys}
If an ergodic system $(X,\mu,T)$ and measurable functions $f_0,f_1,\ldots,f_{k-1}:X\to [-1,1]$ are such that
\[\limsup_{N\to\infty}\frac{1}{N}\sum_{n=1}^N\int_X f_0\cdot (f_1\circ T^n)\cdot \cdots \cdot (f_{k-1}\circ T^{(k-1)n})\,\d\mu \geq \gamma > 0,\]
then there is a factor map $\pi:(X,\mu,T)\to (G/\G,m,R_g)$ onto an ergodic $(k-2)$-step nilrotation such that
\[\|\sfE(f_i\,|\,\pi)\|_2 \geq \frac{1}{2}\gamma \quad\hbox{for each}\ i=0,1,\ldots,k-1.\]

In particular, if $(X\supset A,\mu,T)$ is a process with $\mu(A) \geq \delta$ but no $k$-APs in its return times then there is such a factor map $\pi$ for which
\[\|\sfE(A\,|\,\pi) - \mu(A)\|_2 \geq \frac{1}{2k}\delta^k.\]
\end{cor}

\begin{proof} By Theorem~\ref{thm:HK}, the averages in question have the same asymptotic behaviour as the averages
\[\frac{1}{N}\sum_{n=1}^N\int_X f_0'\cdot (f_1'\circ T^n)\cdot \cdots \cdot (f_{k-1}'\circ T^{(k-1)n})\,\d\mu\]
with $f_i' := \sfE(f_i\,|\,\pi_{k-2})$, and now all these functions still lie in the unit ball of $L^\infty(\mu)$, and so for any $i$ we can apply the Cauchy-Schwartz inequality to $f_i\circ T^{in}$ and the product of the remaining factors to deduce that the above average is bounded in absolute value by $\|f'_i\|_2$.  Since the sum of these averages must be greater than $\gamma/2$ infinitely often, this requires that $\|f_i'\|_2 > \gamma/2$ for each $i$; finally, since $\pi_{k-2}$ is generated by further factor maps onto finite-dimensional $(k-2)$-step nilrotations, letting $\pi$ be a large enough one of these gives the first conclusion.

To derive the second conclusion, first use Theorem~\ref{thm:HK} to obtain
\[\frac{1}{N}\sum_{n=1}^N \mu(A\cap T^{-n}A\cap \cdots \cap T^{-(k-1)n}A)
\sim \frac{1}{N}\sum_{n=1}^N\int_X \prod_{i=0}^{k-1} \sfE(A\,|\,\pi_{k-2})\circ T^{in}\,\d\mu\]
as $N\to\infty$, so that if $A$ contains no $k$-APs in its return times then both of these expressions must vanish as $N\to\infty$. 
Now use the decomposition \[\sfE(A\,|\,\pi_{k-2}) = (\sfE(A\,|\,\pi_{k-2}) - \mu(A)) + \mu(A)\]
to form the telescoping sum
\begin{eqnarray*}
&&\frac{1}{N}\sum_{n=1}^N\int_X \prod_{i=0}^{k-1} \sfE(A\,|\,\pi_{k-2})\circ T^{in}\,\d\mu\\
&& =\sum_{i=0}^{k-1} \frac{1}{N}\sum_{n=1}^N \int_X \mu(A)^i\cdot ((\sfE(A\,|\,\pi_{k-2})\circ T^{in} - \mu(A))\\
&&\quad\quad\quad\quad\quad\quad\quad\quad\quad\quad\quad\quad\quad\quad\quad\quad  \cdot \prod_{i < j \leq k-1}(\sfE(A\,|\,\pi_{k-2})\circ T^{jn})\,\d\mu\\
&&\quad+ \mu(A)^k.
\end{eqnarray*}
Since this tends to $0$, the first $k$ terms of the sum must asymptotically cancel the term $\mu(A)^k \geq \delta^k$, and hence at least one of these first terms must have magnitude arbitrarily close to $\frac{1}{k}\delta^k$ for infinitely many $N$.  The first part of the corollary therefore gives some $(k-2)$-step nilrotation factor $\pi$ for which
\[\|\sfE(A\,|\,\pi) - \mu(A)\|_2 \geq \frac{1}{2k}\delta^k,\]
as required.
\end{proof}

Within Roth's and Gowers' works lie quantitative analogs of the above result: this is what drives Gowers' proof that a failure of uniformity of degree $k-2$ for a density-$\delta$ set $E \subset \{1,2,\ldots,N\}$ gives a partition of $\{1,2,\ldots,N\}$ into fairly long subprogressions on which $E$ enjoys an enlarged relative density (Theorem 18.1 in~\cite{Gow01}).

Heuristically, Gowers shows first that a failure of uniformity of degree $k-2$ implies a nontrivial correlation between $1_E - \delta$ and a function on $\{1,2,\ldots,N\}$ which behaves like the exponential of $\rm{i}$ times a real polynomial of degree $(k-2)$ on many large subprogressions of $\{1,2,\ldots,N\}$.  He then converts this correlation into the desired partition of $\{1,2,\ldots,N\}$ into long subprogressions.  This correlation with a function that behaves `locally' like a degree-$(k-2)$ polynomial is the analog of having a nontrivial conditional expectation onto a $(k-2)$-step nilsystem.  The exact formulation of the finitary `inverse theorem' for the failure of higher-degree uniformity is rather complicated, and we omit it here, but again a gentle introduction with many further references can be found in the book~\cite{TaoVu06} of Tao and Vu. 

In the infinitary ergodic-theoretic setting the implication of Corollary~\ref{cor:corn-on-nilsys} by Theorem~\ref{thm:HK} can easily be reversed: given any indicator function $1_A$ we can decompose it as $1_A = (1_A - \sfE(A\,|\,\pi_{k-2})) + \sfE(A\,|\,\pi_{k-2})$, and now if we form a telescoping sum for the expression $\frac{1}{N}\sum_{n=1}^N \mu(A\cap T^{-n}A\cap \cdots \cap T^{-(k-1)n}A)$ similar to the above then (the contrapositive of) Corollary~\ref{cor:corn-on-nilsys} implies that all the terms involving $1_A - \sfE(A\,|\,\pi_{k-2})$ must vanish as $N\to\infty$, leaving us with Theorem~\ref{thm:HK}:
\[\frac{1}{N}\sum_{n=1}^N \mu(A\cap T^{-n}A\cap \cdots \cap T^{-(k-1)n}A) \sim \frac{1}{N}\sum_{n=1}^N \int_X \prod_{i = 0}^{k-1}\sfE(A\,|\,\pi_{k-2})\circ T^{in}\,\d\mu.\]

However, difficulties emerge when one tries to develop a quantitative analog of this reverse implication, and so provide a truer analog of Theorem~\ref{thm:HK} in the finitary setting. In order to make sense of either conditional expectations such as $\sfE(A\,|\,\pi_{k-2})$, or of the structure of $\pi_{k-2}$ as an inverse limit of a possibly-infinite collection of nilrotation factors, one needs a quantitative analog of taking a limit in $L^2(\mu)$. In practice this leads to an explosion in the bounds obtained.  Although something in this vein is possible (see Tao~\cite{Tao06}), in general it is much less efficient than the density-increment strategy, for which (the finitary analog of) Corollary~\ref{cor:corn-on-nilsys} suffices.  On the other hand, our quick presentation above of the deduction of multiple recurrence from Theorems~\ref{thm:HK} and~\ref{thm:nilsys-MRT} clearly uses the full strength of Theorem~\ref{thm:HK}, and so if instead we started from Corollary~\ref{cor:corn-on-nilsys} it would require us to prove Theorem~\ref{thm:HK} first before proceeding as above.  In the next subsection we will see that the density-increment strategy, by contrast, uses only the conjunction of Corollary~\ref{cor:corn-on-nilsys} and Theorem~\ref{thm:nilsys-MRT}, and it is this feature that accounts for its greater efficiency (leading to better bounds) in the finitary world.

\subsection{The density-increment proof}\label{subs:intro-dens-inc}

The strategy here is to prove Theorem~\ref{thm:MRT} or~\ref{thm:multiMRT} by `induction on $\mu(A)$'.  The technical result underlying this is an ergodic-theoretic analog of Proposition~\ref{prop:dens-inc}.

\begin{prop}[Ergodic-theoretic density-increment]\label{prop:erg-dens-inc}
For each $k\geq 1$ there is a function $c_k:(0,1]\to (0,1]$ that is bounded away from $0$ on compact subsets such that the following holds: if $(X\supset A,\mu,T)$ is a process with $\mu(A) = \delta > 0$ but no $k$-APs in its return times, then for every $\eps > 0$ and $N\geq 1$ there are some non-negligible $B \subset X$ and integer $r\geq 1$ such that
\[\mu(B \triangle T^r B) < \eps\mu(B)\]
and
\[\mu(A\,|\,T^{-rn}B) \geq \delta + c_k(\delta)\quad\quad\hbox{for all}\  -N \leq n \leq N.\]
\end{prop}

Before proving this result let us see why it implies Theorem~\ref{thm:MRT}.

\begin{cor}\label{cor:erg-dens-inc}
With $c_k$ as in Proposition~\ref{prop:erg-dens-inc} the following holds: if there exists a process $(X\supset A,\mu,T)$ having $\mu(A) = \delta > 0$ but no $k$-APs in its return times, then then is another process $(Y\supset B,\nu,S)$ having $\nu(B) \geq \delta + c_k(\delta)$ but no $k$-APs in its return times.
\end{cor}

\begin{proof}[Proof of Corollary from Proposition~\ref{prop:erg-dens-inc}]
By Proposition~\ref{prop:erg-dens-inc}, for any $N$ we can find a non-negligible $B_N \subset X$ and $r_N \geq 1$ such that
\[\mu(B_N \triangle T^{r_N}B_N) < \frac{\mu(B_N)}{N}\]
and
\[\mu(A\,|\,T^{-r_Nn}(B_N)) \geq \delta + c_k(\delta)\quad\quad\hbox{for all}\ -N \leq n \leq N.\]
Let $\nu_N$ be the probability measure on $Y = \{0,1\}^\bbZ$ that is the law of the random sequence
\[\varphi_N:x\mapsto (1_A(T^{r_Nn}(x)))_{n\in \bbZ}\]
for $x$ drawn at random from $\mu(\,\cdot\,|\,B_N)$ (that is, $x$ is chosen `uniformly from $B_N$').

Let $S:Y\to Y$ be the coordinate left-shift and
\[A_a := \{(\o_i)_{i\in\bbZ}:\ \o_a = 1\} \subset Y\quad\quad\hbox{for}\ a \in \bbZ,\]
so $A_a = S^{-a}(A_0)$. The lower bound on the measures $\mu(A\,|\,T^{-r_Nn}(B_N))$ for $-N \leq n \leq N$ implies that any vague accumulation point $\nu$ of the sequence $\nu_N$, say $\nu = \lim_{i\to\infty}\nu_{N_i}$, must satisfy
\[\nu(A_a) = \lim_{i\to\infty} \mu(A\,|\,T^{r_{N_i}a}(B_{N_i})) \geq \delta + c_k(\delta)\quad\quad\forall a.\]
On the other hand, the assumption that there are no $k$-APs in the return times of $A$ implies that
\begin{multline*}
\nu_N(A_a\cap A_{a+r}\cap \cdots \cap A_{a + (k-1)r})\\=
\nu_N\{(\o_i)_{i\in\bbZ}:\ \o_a = \o_{a + r} = \cdots = \o_{a + (k-1)r} = 1\} = 0
\end{multline*}
for all $a \in \bbZ$, $r \geq 1$ and all $N$, and so the same is true for $\nu$.

Finally, the inequality $\mu(B_N \triangle T^{r_N}B_N) < \mu(B_N)/N$ implies for any Borel $C \subset Y$ that
\begin{eqnarray*}
|\nu_N(C) - \nu_N(S^{-1}C)| &=& |\mu(\varphi_N^{-1}C\,|\,B_N) - \mu(T^{-r_N}\varphi_N^{-1}C\,|\,B_N)|\\ &=& |\mu(\varphi_N^{-1}C\,|\,B_N) - \mu(\varphi_N^{-1}C\,|\,T^{r_N}B_N)|\\ &\leq& \frac{\mu(\varphi_N^{-1}C\cap (B_N\triangle T^{r_N}B_N))}{\mu(B_N)} < 1/N,
\end{eqnarray*}
so the vague limit $\nu$ is also $S$-invariant.  Letting $B:= A_0$, this gives a process $(Y \supset B,\nu,S)$ with no $k$-APs in its return times and the desired improved bounds.
\end{proof}

\begin{proof}[Proof of Theorem~\ref{thm:MRT} from Corollary~\ref{cor:erg-dens-inc}]
\quad\emph{Step 1}\quad If $(X\supset A,\mu,T)$ is any counterexample to Theorem~\ref{thm:MRT} with $\delta_0 := \mu(A) > 0$, then a simple vague limit argument can enhance it to an example with the same density value $\delta_0$ but no $k$-APs in its return times.  This construction forms the bulk of this proof.

We first transfer our initially-given example onto the space $Y := \{0,1\}^\bbZ$ with the left-shift $S$. Let
\[B := \{(\omega_i)_{i\in\bbZ} \in Y:\ \omega_0 = 1\},\]
and now consider the map
\[f_A:X\to Y:x\mapsto \big(1_A(T^i(x))\big)_{i\in\bbZ}.\]
This intertwines $T$ with $S$, so the
pushforward $\nu_1 := (f_A)_\#\mu$ is an $S$-invariant Borel measure
on $Y$ for which $\nu_1 (B) =\mu(A)$ and
\begin{multline*}
\liminf_{N\to\infty}\frac{1}{N}\sum_{n=1}^N \nu_1(B\cap
S^{-n}B\cap \cdots\cap S^{-(k-1)n}B)\\ =
\liminf_{N\to\infty}\frac{1}{N}\sum_{n=1}^N \mu(A\cap T^{-n}A\cap\cdots\cap T^{-(k-1)n}A) = 0.
\end{multline*}
This implies that the corresponding averages along any subset $a\cdot \bbN \subset \bbN$, $a\neq 1$, also tend subsequentially to zero:
\[\liminf_{N\to\infty}\frac{1}{N}\sum_{n=1}^N \nu_1(B\cap
S^{-an}B\cap \cdots\cap S^{-a(k-1)n}B) = 0,\]
because for large $N$ the terms corresponding to $n \in a\cdot \bbN$ account for about $1/a$ of the full average.

Now let $\nu_k$ be the image measure of $\nu_1$ under the coordinate-dilation
\[\rm{dil}_k:Y\to Y:(\o_i)_{i\in\bbZ} \mapsto (\o_{ki})_{i\in\bbZ},\]
so that each $\nu_k$ is still shift-invariant and satisfies $\nu_k (B) = \nu_1 (B)$, since $B = \rm{dil}_k^{-1}(B)$ because it depends only on the zeroth coordinate.  Letting $\nu$ be any limit point of the averages $\frac{1}{K}\sum_{k=1}^K\nu_k$ in the vague topology, the above convergence tells us that
\[\nu(B\cap S^{-a}B\cap \cdots\cap S^{-(k-1)a}B) = 0\]
whenever $a\neq 0$.

\quad\emph{Step 2}\quad Having made this simplification, Corollary~\ref{cor:erg-dens-inc} gives a new counterexample with density $\delta_1 \geq \delta_0 + c_k(\delta_0)$.  Since $c_k$ is bounded away from $0$ on the subinterval $[\delta_0,1] \subset (0,1]$, after finitely many iterations this procedure gives a counterexample with density greater than $1$, and hence a contradiction.
\end{proof}

Before presenting the proof of Proposition~\ref{prop:erg-dens-inc} we need one further enabling result, for which we will make our appeal to Theorem~\ref{thm:nilsys-MRT}.  From that theorem we need the consequence that one can approximately decompose an arbitrary positive-measure $U \subset G/\G$ into a collection of almost-invariant sets for different powers of $R_g$.

\begin{prop}\label{prop:rec-in-nilsys}
If $R_g \actson (G/\G,m)$ is an ergodic nilrotation, $U \subset G/\G$ is measurable and of positive measure and $K\geq 1$, then there is a countable set of pairs $\{ (V_1,r_1), (V_2,r_2), \ldots\}$ (which could be finite or infinite) such that
\begin{itemize}
\item[(i)] each $V_i$ has positive measure;
\item[(ii)] if $i \neq i'$ then the unions $\bigcup_{k=-K}^Kg^{r_i k}V_i$ and $\bigcup_{k=-K}^Kg^{r_{i'} k}V_{i'}$ are disjoint;
\item[(iii)] $U\supset \bigcup_{i\geq 1}\bigcup_{k=-K}^K g^{r_i k}V_i$;
\item[(iv)] and $m\big(U\setminus \bigcup_{i\geq 1}\bigcup_{k=-K}^K g^{r_ik}V_i\big) = 0$.
\end{itemize}
\end{prop}

\begin{proof}\quad The proof given here invokes Zorn's Lemma, although a more careful argument shows that it only really needs the ability to induct transfinitely below $\omega_1$.  Frustratingly, I have not been able to find a proof that avoids this kind of induction entirely, although in the finitary analog of this step all sets are finite and so the issue does not arise.

Let $\cal{A}$ be the set of all countable families of pairs $\{(V_1,r_1), (V_2,r_2), \ldots\}$ that have properties (i--iii) above (but possibly not (iv)), and order $\cal{A}$ by inclusion of families. If $\cal{F} = \{(V_1,r_1),(V_2,r_2),\ldots\} \in \cal{A}$ then set
\[m(\cal{F}) := m\Big(\bigcup_{i\geq 1}\bigcup_{k=-K}^Kg^{r_ik}V_i\Big).\]

Since $m(U) > 0$, Theorem~\ref{thm:nilsys-MRT} promises some $r$ such that
\[m(g^{-Kr}U\cap g^{-(K-1)r}U\cap\cdots\cap g^{Kr}U) > 0.\]
Therefore the set $V := g^{-Kr}U\cap g^{-(K-1)r}U\cap\cdots\cap g^{Kr}U$ has positive measure and satisfies $g^{rk}V \subset U$ for every $-K \leq k \leq K$, so $\{(V,r)\} \in \cal{A}$ and hence $\cal{A}$ is nonempty.

Now suppose that $(\cal{F}_\a)_\a$ is a totally ordered family in $\cal{A}$.  Since $m(V) > 0$ for any $(V,r) \in \cal{F}_\a$, the values of the measures $m(\cal{F}_\a)$ are totally ordered, are all distinct and are bounded by $1$.  We may therefore extract a non-decreasing sequence $\cal{F}_{\a_1} \subset \cal{F}_{\a_2} \subset \cdots$ such that $m(\cal{F}_{\a_i}) \to \sup_\a m(\cal{F}_\a)$ as $i\to\infty$.

Since each $\cal{F}_\a$ is countable and they are totally ordered, it follows that $\cal{G} := \bigcup_{i\geq 1}\cal{F}_{\a_i}$ is still countable, and in fact is still a member of $\cal{A}$.  Moreover, if $(V,r) \in \cal{F}_\a$ for some $\a$, then this pair must actually appear in some $\cal{F}_{\a_i}$, for otherwise we would have $m(\cal{F}_{\a_i}) \leq m(\cal{F}_\a) - m\big(\bigcup_{k=-K}^Kg^{rk}V\big)$ for every $i$, contradicting our construction.  Hence $\cal{G} \supset \cal{F}_\a$ for all $\a$, and so $\cal{G}$ is an upper bound for the chain $(\cal{F}_\a)_\a$.

Therefore by Zorn's Lemma the whole family $\cal{A}$ has a maximal element, say $\cal{F} = \{(V_1,r_1), (V_2,r_2), \ldots\}$.  Now we need simply observe that this must have $m(\cal{F}) = m(A)$ (which implies property (iv)), since otherwise another appeal to Theorem~\ref{thm:nilsys-MRT} would give $V' \subset U\setminus \bigcup\cal{F}$ and $r'\geq 1$ such that $\cal{F}\cup \{(V',r')\} \in \cal{A}$, contradicting the maximality of $\cal{F}$.  Therefore $\cal{F}$ has all the desired properties, and the proof is complete.
\end{proof}

\begin{rmk}
The finitary analog of this result in~\cite{Gow01} (see his Corollary 5.6) is very elementary and quantitative.  I suspect that a version of Gowers' proof could be adapted to the present setting (perhaps with some additional assumptions on $U$, such as that it be open with piecewise-smooth boundary), but that this would require the use of a Mal'cev basis for $G$ and the ability to study orbits of $R_g$ in terms of `explicit' generalized polynomials using the resulting coordinate system.  Such a more quantitative argument would probably be considerably longer than the proof given above. \fin
\end{rmk}

We can now complete the density-increment proof of multiple recurrence using the above proposition and Corollary~\ref{cor:corn-on-nilsys}.

\begin{proof}[Proof of Proposition~\ref{prop:erg-dens-inc}]
Suppose that $(X\supset A,\mu,T)$ is a process having $\mu(A) =: \delta > 0$ but no $k$-APs in its return times.  Then Corollary~\ref{cor:corn-on-nilsys} gives
\[\|\sfE(A\,|\,\pi) - \mu(A)\|_2 \geq \frac{1}{2k}\delta^k\]
for some factor map $\pi:(X,\mu,T)\to (G/\G,m,R_g)$ onto a $(k-2)$-step nilrotation.

Since $G/\G$ is compact and its Borel $\s$-algebra is generated by its open sets, we can find a finite Borel partition $\cal{U}$ of $G/\G$ into small-diameter positive-measure pieces such that
\[\frac{1}{m(U)}\int_U|\sfE(A\,|\,\pi) - \mu(A\,|\,\pi^{-1}U)|\,\d m < \frac{1}{20k}\delta^k\]
for all $U \in \cal{U}\setminus \cal{U}_{\rm{bad}}$, where $\cal{U}_{\rm{bad}}$ is a subcollection such that $m(\bigcup\cal{U}_{\rm{bad}}) < \delta^k/20k$.  Combined with the preceding inequality, this implies that there is some $U \in \cal{U}$ for which
\[\mu(A\,|\,\pi^{-1}U) > \delta + \frac{1}{10k}\delta^k.\]

Now given $N\geq 1$ choose $K := LN\geq 1$ with $L\geq 1$ so large that $1/L \leq \delta^k/20k$ and $L > 2/\eps + 1$. Apply Proposition~\ref{prop:rec-in-nilsys} to the set $U$ to obtain pairs $(V_1,r_1)$, $(V_2,r_2)$, \ldots with each $V_i$ having positive measure and such that the unions $\bigcup_{k=-K}^Kg^{r_ik}V_i$ for $i\geq 1$ are pairwise disjoint, all contained in $U$ and together fill up $m$-almost all of $U$.
In view of the convex combination
\[\mu(A\,|\,\pi^{-1}U) = \sum_{i \geq 1} \frac{m\big(\bigcup_{k=-K}^Kg^{r_ik}V_i\big)}{m(U)}\mu\Big(A\,\Big|\,\pi^{-1}\Big(\bigcup_{k=-K}^K g^{r_ik}V_i\Big)\Big),\]
there is some $i$ for which
\[\mu\Big(A\,\Big|\,\pi^{-1}\Big(\bigcup_{k=-K}^Kg^{r_ik}V_i\Big)\Big) \geq \delta + \frac{1}{10k}\delta^k.\]

Letting $B' = \bigcup_{k=-K}^{K-2N-1}g^{r_ik}V_i$ and 
\[C := \bigcup_{k=-K}^K g^{r_ik} V_i \Big\backslash B',\]
the shifts $C$, $g^{-r_i(2N+1)}C$, $g^{-2r_i(2N+1)}C$, \ldots, $g^{-(L-1)r_i(2N+1)}C$ are pairwise disjoint and contained in $\bigcup_{k=-K}^Kg^{r_ik}V_i$, so each has measure at most $\frac{1}{L}m\big(\bigcup_{k=-K}^K g^{r_ik}V_i\big)$ and therefore
\[m(B') \geq \frac{L-1}{L}m\Big(\bigcup_{k=-K}^K g^{r_ik}V_i\Big).\]

In addition, the set difference $g^{r_i}B'\setminus B'$ is contained in $C$ and so has measure at most
\[\frac{1}{L}m\Big(\bigcup_{k=-K}^K g^{r_ik}V_i\Big) \leq \frac{1}{L-1}m(B'),\]
and a symmetrical argument controls the measure of $B' \setminus g^{r_i}B'$ so together we obtain
\[m(B'\triangle g^{r_i}B') \leq \frac{2}{L-1}m(B') < \eps m(B').\]

Finally, letting $B := g^{-r_iN}B'$, it follows that $g^{r_in}B \subset \bigcup_{k=-K}^Kg^{r_ik}V_i$ for all $-N \leq n \leq N$, that
\begin{eqnarray*}
\mu(A\,|\,\pi^{-1}g^{r_in}B) &\geq& \frac{\mu(A\cap \pi^{-1}g^{r_in}B)}{m\big(\bigcup_{k=-K}^Kg^{r_ik}V_i\big)}\\ &\geq& m\Big(A\,\Big|\,\pi^{-1}\Big(\bigcup_{k=-K}^Kg^{r_ik}V_i\Big)\Big) - \frac{m(C)}{m\big(\bigcup_{k=-K}^Kg^{r_ik}V_i\big)}\\ &\geq& \delta + \frac{1}{10k}\delta^k - \frac{1}{L}\\ &\geq& \delta + \frac{1}{20k}\delta^k,
\end{eqnarray*}
and that $B$ enjoys the same approximate $g^{r_i}$-invariance as $B'$, completing the proof of Proposition~\ref{prop:erg-dens-inc} with $c_k(\delta) := \frac{1}{20k}\delta^k$.
\end{proof}

\section{Two commuting transformations}\label{sec:2D}

\subsection{The density increment in higher dimensions}

With the appearance of Gowers' density-increment proof of
Szemer\'edi's Theorem, it became natural to
ask whether a similar approach can yield improved upper bounds
for any cases of the multidimensional Szemer\'edi Theorem.  Gowers discusses this question explicitly in~\cite{Gow00}.  It poses significant new challenges, and remains mostly open.  For the analogous ergodic-theoretic study of multiple recurrence we will see that the difficulty arises from the nature of the characteristic factors in multiple dimensions, which are rather more complicated than the pro-nilsystems that give the complete picture for powers of a single ergodic transformation.

In the context of finitary proofs, it is still possible to set up a `directional' variant of the norms (actually now just seminorms) that Gowers introduced to define uniformity, and to show that the resulting new notion of uniformity does control the count of the desired patterns in a subset $E \subset \{1,2,\ldots,N\}^d$.  The difficulty is in handling those sets, or more generally functions $f:\{1,2,\ldots,N\}^d \to [-1,1]$, which are \emph{not} uniform in the sense of this seminorm.  Extending the approach of Roth and Gowers requires one to find the appropriate class of functions against which an arbitrary function  must see a large correlation if it is not uniform.  For uniformity of  degree $k$ in the one-dimensional setting, these were the functions which on many long arithmetic subprogressions of $\{1,2,\ldots,N\}$ agree with the exponential of $\rm{i}$ times some degree-$k$ real polynomials (see the discussion following Corollary~\ref{cor:corn-on-nilsys}), but in the multi-dimensional setting they are much more complicated.  Part of the difficulty in extending Gowers' approach lies in the problem of identifying the most appropriate class of functions to use here, and part of it lies in establishing some necessary properties of those functions once they have been found (properties which are fairly classical in the case of the one-dimensional `local' polynomial functions).

However, in spite of these difficulties, Gowers-like bounds have now been obtained in the following special case of Theorem~\ref{thm:multiSzem} by
Shkredov:

\begin{thm}
There is some absolute constant $C > 0$ such that if $\delta
> 0$, $N \geq 2^{2^{2^{1/\delta^C}}}$ and $A \subset \{1,2,\ldots,N\}^2$
has $|A|\geq \delta N^2$, then $A$ contains a corner:
\[A \supseteq \{\bf{a},\bf{a} + r\bf{e}_1,\bf{a} + r\bf{e}_2\}\]
for some $\bf{a} \in \{1,2,\ldots,N\}^2$ and $r \geq 1$, where
$\bf{e}_1,\bf{e}_2$ are the standard basis vectors in $\bbZ^2$.
\end{thm}

In fact, since the appearance of his original article~\cite{Shk06},
in~\cite{Shk06-b} Shkredov has improved the above bound further to the form $2^{2^{1/\delta^C}}$,
effectively by replacing a repeated descent to arithmetic
subprogressions with a descent through a nested sequence of Bohr
sets, following Bourgain's use of these for his improved bounds in
Roth's Theorem~\cite{Bou99,Bou08}.  In addition, Shkredov has shown
in~\cite{Shk09} how this latter argument can also be implemented in
the setting of arbitrary finite Abelian groups (see also Section 5 of Green's
survey~\cite{Gre05} for a treatment of the case of high-dimensional vector
spaces over a finite field).  However, for the sake of simplicity this note will focus on analogs of the original paper~\cite{Shk06}, and
where appropriate make comparisons to the steps taken there.

Thus, we here present a new proof of the following special case of
Theorem~\ref{thm:multiMRT}:

\begin{thm}\label{thm:Shk}
If $T_1,T_2:\bbZ\actson(X,\mu)$ commute and $A \subset X$ has $\mu(A) > 0$ then
\[\liminf_{N\to\infty}\frac{1}{N}\sum_{n=1}^N \mu(A\cap T_1^{-n}A\cap T_2^{-n}A) > 0.\]
\end{thm}

Henceforth we will generally refer to the quadruple $(X,\mu,T_1,T_2)$ as a \textbf{$\bbZ^2$-system}, in reference to the action of the whole group generated by $T_1$ and $T_2$.

In contrast with our work in the previous section, the analog of Theorem~\ref{thm:HK} that will appear in this setting does not reduce our study to a class of systems for which multiple recurrence can simply be proved directly, as was the case using Theorem~\ref{thm:nilsys-MRT}.  For this reason, although Theorem~\ref{thm:Shk} has of course been known since Furstenberg and Katznelson's work, the proof presented here is not quite so redundant as is the density-increment proof in one dimension (recall the discussion following the statement of Theorem~\ref{thm:nilsys-MRT}).

An important aspect of Shkredov's proof is the introduction, in addition to $E \subset \{1,2,\ldots,N\}^2$, of a
superset of it which is a product set $F_1\times F_2$ which must also be manipulated as the proof proceeds.  We will employ a similar idea in the following, where for a system $(X,\mu,T_1,T_2)$ the structure of a `product set' is replaced by that of an intersection of sets which are invariant under either $T_1$ or $T_2$.  The importance of these special sets corresponds to the emergence of the factor generated by the $T_1$- or $T_2$-invariant sets within the structure of the characteristic factors.  With this in mind, we make the following analog of Definition~\ref{dfn:process}.

\begin{dfn}[Augmented process]\label{dfn:aug-process}
An \textbf{augmented process} is a $\bbZ^2$-system $(X,\mu,T_1,T_2)$ together with distinguished measurable subsets $A$, $E_1$ and $E_2$ satisfying $A\subset E_1\cap E_2$ and such that $E_i$ is $T_i$-invariant.  We shall sometimes denote these data by $(X\supset E_1\cap E_2\supset A,\mu,T_1,T_2)$.
\end{dfn}

\begin{dfn}
An augmented process has \textbf{no corners in its return set} if
\[\mu(A\cap T_1^{-n}A\cap T_2^{-n}A) = 0\quad\forall n\neq 0.\]
\end{dfn}

In addition, the following notation will be used throughout the sequel.

\begin{dfn}[Partially invariant sets]
If $(X,\mu,T_1,T_2)$ is a $\bbZ^2$-system, then a subset $A \subset X$ is \textbf{partially invariant} if it is invariant under $T_1^{n_1}T_2^{n_2}$ for some $(n_1,n_2)$.  The $\s$-algebra of $(T^{n_1}_1T_2^{n_2})$-invariant measurable sets is denoted by $\S^{(n_1,n_2)}$, and in addition we let $\zeta_0^{(n_1,n_2)}$ be some factor map $X\to Z_0^{(n_1,n_2)}$ onto an auxiliary system where the transformation in direction $(n_1,n_2)$ is trivial and which generates $\S^{(n_1,n_2)}$.
\end{dfn}

(This correspondence between globally invariant $\s$-subalgebras of $\S$ and factor maps onto other systems is standard in ergodic theory; see, for instance, Chapter 2 of~\cite{Aus--thesis} and the references given there.)

\begin{dfn}[Kronecker factors]
If $(X,\mu,T_1,T_2)$ is an ergodic $\bbZ^2$-system then $\zeta_1^T$ will denote some choice of a factor map from $X$ onto an action by rotations on a compact Abelian group which generates the Kronecker factor of $(X,\mu,T_1,T_2)$, and similarly for $\bbZ$-systems.
\end{dfn}

\begin{dfn}[Arithmetic of factors]
Given two factor maps $\pi_i:(X,\mu,T_1,T_2)\to (Y_i,\nu_i,S_{1,i},S_{2,i})$ of a $\bbZ^2$-system, we let $\pi_1\vee \pi_2$ denote a factor map which generates the same $\s$-algebra as $\pi_1$ and $\pi_2$ together (for example, the Cartesian product map $(\pi_1,\pi_2):X\to Y_1\times Y_2)$ will do), and $\pi_1\wedge \pi_2$ denote a factor map which generates the $\s$-algebra of all sets that are both $\pi_1$- and $\pi_2$-measurable.
\end{dfn}

In his setting, Shkredov considered nested inclusions
\[E \subset F_1\times F_2 \subset \{1,2,\ldots,N\}^2.\]
His main innovation is the result that in order to count approximately the number of corners in $E$ it suffices to control the non-uniformity of $E$ relative to its superset $F_1\times F_2$, and crucially \emph{to an extent which depends only on the relative density $\frac{|E|}{|F_1||F_2|}$}, provided the sets $F_1$ and $F_2$ have some
uniformity properties of their own. He effectively formulated
this latter uniformity condition in terms of a uniform bound on the
one-dimensional Fourier coefficients of the $F_i$, but for our sets $E_i \in \S^{T_i}$ it turns out that a stronger condition is more convenient, formulated in terms of the independence of their shifts under $T_j$ for $j\neq i$; this condition will appear shortly.

The need for the $E_i$ below becomes natural upon understanding the analog of Theorem~\ref{thm:HK} for the averages of Theorem~\ref{thm:Shk}.  However, in the ergodic theoretic world this involves another new twist, which has no real analog in the finitary setting.  It turns out that simply-described characteristic factors for the averages of Theorem~\ref{thm:Shk} may be obtained only after ascending to some extension of the initially-given system. (The original system will certainly \emph{have} characteristic factors, but they may be much more complicated to describe.)  Of course, it suffices to prove multiple
recurrence for such an extension, and so this is quite adequate for our proof strategy.  The following result is specialized from the
construction of so-called `pleasant and isotropized extensions'
in~\cite{Aus--nonconv,Aus--newmultiSzem}.

\begin{thm}\label{thm:Fberg-struct}
Any $\bbZ^2$-system $(X^\circ,\mu^\circ,T_1^\circ,T_2^\circ)$ has an extension
\[\pi:(X,\mu,T_1,T_2)\to (X^\circ,\mu^\circ,T^\circ_1,T^\circ_2)\]
with the property that
\begin{multline*}
\frac{1}{N}\sum_{n=1}^N\int_X f_0\cdot (f_1\circ T_2^n)\cdot(f_2\circ T_2^n)\,\d\mu\\
\sim \frac{1}{N}\sum_{n=1}^N\int_X \sfE_\mu(f_0\,|\,\pi_0)\cdot \sfE_\mu(f_1\,|\,\pi_1)\cdot \sfE_\mu(f_2\,|\,\pi_2)\,\d\mu
\end{multline*}
as $N\to\infty$ for any $f_0,f_1,f_2 \in L^\infty(\mu)$, where
\begin{eqnarray*}
\pi_0 &:=& \zeta_0^{(1,0)}\vee \zeta_0^{(0,1)}\\
\pi_1 &:=& \zeta_0^{(1,0)}\vee \zeta_0^{(1,-1)}\\
\pi_2 &:=& \zeta_0^{(1,-1)}\vee \zeta_0^{(0,1)}.
\end{eqnarray*}
\qed
\end{thm}

\begin{dfn}[Pleasant system]
Essentially following the nomenclature of~\cite{Aus--nonconv}, we will refer to a system having the property of the extension constructed above as \textbf{pleasant}.
\end{dfn}

Replacing an initially-given $\bbZ^2$-system with an extension if necessary, we
may henceforth concentrate on pleasant systems.

With this description of the characteristic factors in hand, we can now offer our ergodic theoretic translation of Shkredov's main estimate (Theorem 7 in~\cite{Shk06}).

\begin{prop}\label{prop:Shk-main-estimate}
Suppose that $(X\supset E_1\cap E_2\supset A,\mu,T_1,T_2)$ is a pleasant augmented process with
$\mu(A) > 0$, that
\begin{itemize}
\item the return-set of $A$ contains no corners, and
\item $E_1\perp T_2^n(E_1)$ and $E_2\perp T_1^n(E_2)$ for all $n\neq 0$, where $\perp$ denotes independence,
\end{itemize}
and let $\pi_0 := \zeta_0^{(1,0)}\vee \zeta_0^{(0,1)}$. Then
\[\|\sfE_\mu(A\,|\,\pi_0) - \mu(A\,|\,E_1\cap E_2)\|_{L^2(\mu(\cdot\,|\,E_1\cap E_2))} \geq \mu(A\,|\,E_1\cap E_2)^3.\]
\end{prop}

The benefit of working with the conditions $E_1\perp T_2^n(E_1)$
is that they will be relatively easy to recover for the new process that we construct during the coming density increment.  We will see shortly (Corollary~\ref{cor:self-orthog-implies-orthog-to-Kron}) that this condition implies that $E_1$ is orthogonal to the Kronecker factor $\zeta_1^T$, and this orthogonality is a truer ergodic-theoretic analog of Shkredov's condition that they be degree-$1$ uniformity.

Proposition~\ref{prop:Shk-main-estimate} will be proved in Subsection~\ref{subs:main-est}.

\subsection{A closer look at the characteristic factors and the main estimate}

Before proving
Proposition~\ref{prop:Shk-main-estimate} we need some simple auxiliary results about the factors appearing in Theorem~\ref{thm:Fberg-struct}.

\begin{lem}\label{lem:joint-dist-of-part-invts}
If $(X,\mu,T_1,T_2)$ is ergodic as a $\bbZ^2$-system, then any two of the factors $\zeta_0^{(1,0)}$, $\zeta_0^{(0,1)}$, $\zeta_0^{(1,-1)}$ are independent, and the three together are relatively independent over their intersections with the Kronecker factor:
\[\zeta_1^T \wedge \zeta_0^{(1,0)},\ \zeta_1^T\wedge \zeta_0^{(0,1)},\ \zeta_1^T\wedge \zeta_0^{(1,-1)}.\]
\end{lem}

\begin{proof}
The first assertion is an immediate consequence of the commutativity of $T_1$ and $T_2$.  We prove it for $\zeta_0^{(1,0)}$ and $\zeta_0^{(0,1)}$, the other pairs being similar: since $T_1$ and $T_2$ commute, if $A_1$ is $T_1$-invariant then the conditional expectation $\sfE(A_1\,|\,\zeta_0^{(0,1)})$ is invariant under both $T_1$ and $T_2$ and hence constant, by ergodicity, and must therefore simply equal $\mu(A_1)$.

Handling the three factors together is only a little trickier.  If $A_1 \in \zeta_0^{(1,0)}$, $A_2 \in \zeta_0^{(0,1)}$ and $A_{12} \in \zeta_0^{(1,-1)}$, then by the first assertion the target of the factor map $\zeta_0^{(1,0)}\vee \zeta_0^{(0,1)}$ can simply be identified with a Cartesian product system
\[(Y_1\times Y_2,\nu_1\otimes \nu_2,S_2\times \rm{id},\rm{id} \times S_1)\]
where $S_2$ is an ergodic transformation of the first coordinate alone and $S_1$ an ergodic transformation of the second.  The fact that the invariant measure of this target system is a product $\nu_1\otimes \nu_2$ corresponds to the independence of $\zeta_0^{(1,0)}$ and $\zeta_0^{(0,1)}$.  In this picture the set $A_i$ is lifted from some subset $A_i' \subset Y_i$ under the further coordinate projection $Y_1\times Y_2 \to Y_i$.  Since $A_{12}$ is $\zeta_0^{(1,-1)}$-measurable one has
\[\mu(A_1\cap A_2\cap A_{12}) = \int_X \sfE(A_1\cap A_2\,|\,\zeta_0^{(1,-1)})\cdot 1_{A_{12}}\,\d\mu,\]
and on $Y_1\times Y_2$ the conditional expectation $\sfE(A_1\cap A_2\,|\,\zeta_0^{(1,-1)})$ is identified with the conditional expectation of $A_1'\times A_2'$ onto the sets invariant under $S_2^{-1}\times S_1$.

It is standard that the invariant sets of a product of ergodic systems depend only on the product of their Kronecker factors (see, for instance, the more general Theorem 7.1 in Furstenberg's original paper~\cite{Fur77}), and so our conditional expectation of $A_1'\times A_2'$ is actually onto the invariant sets of $\zeta_1^{S_2}\times \zeta_1^{S_1}$, whose lifts back up to $X$ must all be measurable with respect to $\zeta_1^T$.  Therefore $\sfE(A_1\cap A_2\,|\,\zeta_0^{(1,-1)})$ is actually $\zeta_1^T$-measurable, and so the above integral is equal to
\[\int_X \sfE(A_1\cap A_2\,|\,\zeta_0^{(1,-1)})\cdot \sfE(A_{12}\,|\,\zeta_1^T\wedge \zeta_0^{(1,-1)})\,\d\mu = \int_X 1_{A_1\cap A_2}\cdot \sfE(A_{12}\,|\,\zeta_1^T\wedge \zeta_0^{(1,-1)})\,\d\mu.\]
Applying a symmetric argument to the other sets $A_i$ now shows that this equals
\[\int_X \sfE(A_1\,|\,\zeta_1^T\wedge \zeta_0^{(1,0)})\cdot \sfE(A_2\,|\,\zeta_1^T\wedge \zeta_0^{(0,1)})\cdot \sfE(A_{12}\,|\,\zeta_1^T\wedge \zeta_0^{(1,-1)})\,\d\mu,\]
which is the desired assertion of relative independence.
\end{proof}

\begin{rmk}
The second part of the above lemma, although a very simple consequence of classical results in ergodic theory, has an important counterpart in Lemma 1 (4) of~\cite{Shk06}. It corresponds to the assertion that if sets $F_1,F_2,F_{12}\subseteq \bbZ/N\bbZ$ are lifted through the coordinate projections
\[(n_1,n_2) \mapsto n_1,\ n_2,\ n_1 + n_2\quad\hbox{respectively}\]
and if in addition they are all linearly uniform (meaning that their Fourier coefficients are all small), then their lifts are approximately independent. In his paper Shkredov phrases this in terms of the approximate constancy of a certain convolution of two functions that are lifted from $\bbZ/N\bbZ$ in this way. \fin
\end{rmk}

\begin{lem}\label{lem:weak-rel-weak-mix}
Suppose that $(Y,\nu,S)$ is an ergodic $\bbZ$-system and let $\zeta^S_1:(Y,\nu,S)\to (Z,m_Z,R)$ be its Kronecker factor.  Then for any $f,g \in L^\infty(\nu)$, any $B \subset X$ with $\nu(B) > 0$ that is $\zeta^S_1$-measurable, and any $\eps > 0$, the set
\[\Big\{n \in \bbZ:\ \Big|\int_B f\cdot (g\circ S^n)\,\d\nu - \int_B\sfE_\nu(f\,|\,\zeta^S_1)\cdot\sfE_\nu(g\circ S^n\,|\,\zeta^S_1)\,\d\nu\Big| \leq \eps\Big\}\]
has density $1$ in $\bbZ$.
\end{lem}

\begin{rmk}
The conclusion of this lemma may be re-phrased as asserting that
\[\sfE_\nu(f\cdot (g\circ S^n)\,|\,\zeta^S_1) \sim \sfE_\nu(f\,|\,\zeta^S_1)\cdot\sfE_\nu(g\circ S^n\,|\,\zeta^S_1)\]
\emph{weakly} in $L^2(m_Z)\circ \zeta^S_1 \subset L^2(\nu)$ as $n\to\infty$ along some full-density subset of $\bbZ$.  Strong convergence here for all $f$ and $g$, rather than weak convergence, would be equivalent to $(Y,\nu,S)$ being relatively weakly mixing over its Kronecker factor, which is not always the case. \fin
\end{rmk}

\begin{proof}
On the one hand
\[\int_B f\cdot (g\circ S^n)\,\d\nu = \int_Y (f1_B)\cdot (g\circ S^n)\,\d\nu\]
and on the other $\sfE_\nu(f1_B\,|\,\zeta^S_1) = \sfE_\nu(f\,|\,\zeta^S_1)1_B$, because $B$ is already $\zeta^S_1$-measurable, so after replacing $f$ with $f1_B$ if necessary it suffices to treat the case $B = Y$. The desired assertion is now simply that
\[\langle f,g\circ S^n\rangle \sim \langle \sfE(f\,|\,\zeta^S_1),\sfE(g\,|\,\zeta^S_1)\circ S^n\rangle\]
as $n\to\infty$ outside some zero-density set of `exceptional times' in $\bbZ$, and this is a well-known property of the Kronecker factor (see, for instance, Furstenberg~\cite{Fur81}).
\end{proof}

\begin{cor}\label{cor:self-orthog-implies-orthog-to-Kron}
If $(Y,\nu,S)$ is an ergodic $\bbZ$-system and $E \subset Y$ is
such that $E\perp S^n(E)$ for all $n\neq 0$ then $E$ is independent from the $\s$-algebra generated by $\zeta^S_1$ under $\mu$.
\end{cor}

\begin{proof}
The degenerate case $B = Y$ of the preceding lemma shows that asymptotically for most $n$ we have
\[\nu(E\cap S^{-n}E) \approx \int_Y \sfE_\nu(E\,|\,\zeta_1^S)\cdot (\sfE_\nu(E\,|\,\zeta_1^S)\circ S^n)\,\d\nu.\]
Since the Kronecker factor $(Z,m_Z,R)$ is a compact system, for any $\eps > 0$ there is some nonempty Bohr set in $\bbZ$ along which the right-hand values above return within $\eps$ of
\[\int_Y \sfE_\mu(E\,|\,\zeta_1^S)^2\,\d\nu = \|\sfE_\mu(E\,|\,\zeta_1^S)\|_2^2.\]
This Bohr set must have positive density and therefore contain a further subset of values of $n$ where our first approximation above is also good.  This implies that for any $\eps > 0$ there are infinitely many $n$ for which
\[\big|\nu(E\cap S^n(E)) - \|\sfE_\mu(E\,|\,\zeta_1^S)\|_2^2\big| < \eps,\]
but on the other hand our assumption on $E$ implies that
\[\nu(E\cap S^n(E)) = \nu(E)^2 = \|\sfE_\mu(E\,|\,\zeta_1^S)\|_1^2\quad \forall n\neq 0.\]
This is possible only if $\|\sfE_\mu(E\,|\,\zeta_1^S)\|_1 = \|\sfE_\mu(E\,|\,\zeta_1^S)\|_2$, which in turn requires that $\sfE_\mu(E\,|\,\zeta_1^S)$ be constant, as required.
\end{proof}

\begin{lem}\label{lem:1D-unif-to-rel-ind}
If $(X,\mu,T_1,T_2)$ is an ergodic $\bbZ^2$-system and $E_1 \in
\S^{(1,0)}$, $E_2 \in \S^{(0,1)}$ are such that $E_i\perp T_j^n(E_i)$ for all
$n\neq 0$ whenever $\{i,j\} = \{1,2\}$, then also $E_1$ (resp. $E_2$) is independent from $\zeta_0^{(0,1)}\vee \zeta_0^{(1,-1)}$ (resp. $\zeta_0^{(1,0)}\vee \zeta_0^{(1,-1)}$).
\end{lem}

\begin{rmk}
For us this is analogous to the way Shkredov uses his Lemma 1 to
estimate the second term in equation (21) in his Theorem 7. \fin
\end{rmk}

\begin{proof}
The second part of Lemma~\ref{lem:joint-dist-of-part-invts} implies
\[\sfE_\mu(E_1\,|\,\zeta_0^{(0,1)}\vee \zeta_0^{(1,-1)}) = \sfE_\mu(\sfE_\mu(E_1\,|\,\zeta_1^T\wedge \zeta_0^{(1,0)})\,|\,\zeta_0^{(0,1)}\vee \zeta_0^{(1,-1)}).\]
Corollary~\ref{cor:self-orthog-implies-orthog-to-Kron} now gives that $\sfE_\mu(E_1\,|\,\zeta_1^T\wedge \zeta_0^{(1,0)})$ is constant, and hence so is the conditional expectation of interest.  The proof for $E_2$ is exactly similar.
\end{proof}

\subsection{The main estimate}\label{subs:main-est}

\begin{proof}[Proof of Proposition~\ref{prop:Shk-main-estimate}]
Define the trilinear form $\l$ on $L^\infty(\mu)^3$ by
\[\l(f_0,f_1,f_2) := \lim_{N\to\infty}\frac{1}{N}\sum_{n=1}^N\int_X f_0\cdot (f_1\circ T_1^n)\cdot (f_2\circ T_2^n)\,\d\mu.\]
(In fact this is the integral of the function $f_0\otimes f_1\otimes f_2$ against a certain three-fold self-joining of the system $(X,\mu,T_1,T_2)$ called the `Furstenberg self-joining'.  We will not use that more elaborate formalism here, but refer the reader to~\cite{Aus--thesis} and the references given there for a detailed explanation, as well as a proof that the limit exists.)

Our assumptions include that $\l(A,A,A) = 0$ (where we have simply written $A$ in place of $1_A$),
but on the other hand by Theorem~\ref{thm:Fberg-struct} we have
\begin{eqnarray*}
\l(A,A,A) &=& \l(\sfE(A\,|\,\pi_0), A,A)\\
&=& \l\big(\sfE(A\,|\,\pi_0) - \mu(A\,|\,E_1\cap E_2)1_{E_1\cap E_2}, A, A\big)\\
&&\quad\quad\quad\quad\quad\quad\quad\quad + \mu(A\,|\,E_1\cap E_2)\cdot\l(E_1\cap E_2,A,A).
\end{eqnarray*}

We now estimate these two terms separately.

\emph{First term}\quad Directly from the definition of $\l$ we deduce that
\begin{eqnarray*}
&&\big|\l\big(\sfE(A\,|\,\pi_0) - \mu(A\,|\,E_1\cap E_2)1_{E_1\cap E_2},A,A\big)\big| \\
&&\leq \l\big(|\sfE(A\,|\,\pi_0) - \mu(A\,|\,E_1\cap E_2)1_{E_1\cap E_2}|,A,A\big) \\
&&\leq \l\big(|\sfE(A\,|\,\pi_0) - \mu(A\,|\,E_1\cap E_2)1_{E_1\cap E_2}|,E_1\cap E_2,E_1\cap E_2\big),
\end{eqnarray*}
where the second inequality uses that these three functions are non-negative
and that $1_A \leq 1_{E_1 \cap E_2}$.  Now another appeal to Theorem~\ref{thm:Fberg-struct} shows that this last upper bound is equal to
\[\l\big(|\sfE(A\,|\,\pi_0) - \mu(A\,|\,E_1\cap E_2)1_{E_1\cap E_2}|,\ \sfE(E_1\cap E_2\,|\,\pi_1),\
\sfE(E_1\cap E_2\,|\,\pi_2)\big).\]
From our hypothesis that $E_1\perp T_2^n(E_1)$ for all nonzero $n$ and Lemma~\ref{lem:1D-unif-to-rel-ind} it follows that $E_2$ is $\pi_2$-measurable whereas $E_1$ is independent from $\pi_2$, and hence that
\[\sfE(E_1\cap E_2\,|\,\pi_2) = \mu(E_1) 1_{E_2},\]
and similarly with the two indices reversed. Given this we can re-write the above term as
\begin{eqnarray*}
&&\mu(E_1)\mu(E_2)\l\big(|\sfE(A\,|\,\pi_0) - \mu(A\,|\,E_1\cap E_2)1_{E_1\cap E_2}|,E_1,E_2\big)\\ &&= \mu(E_1)\mu(E_2)\lim_{N\to\infty}\frac{1}{N}\sum_{n=1}^N\int_X |\sfE(A\,|\,\pi_0) - \mu(A\,|\,E_1\cap E_2)1_{E_1\cap E_2}|\\
&&\quad\quad\quad\quad\quad\quad\quad\quad\quad\quad\quad\quad\quad\quad\quad\quad\quad\quad\quad\quad\quad\cdot 1_{T_1^{-n}(E_1)}\cdot 1_{T_2^{-n}(E_2)}\,\d\mu\\
&&=  \mu(E_1)\mu(E_2)\int_X |\sfE(A\,|\,\pi_0) - \mu(A\,|\,E_1\cap E_2)1_{E_1\cap E_2}|\,\d\mu,
\end{eqnarray*}
where for the second equality we have now used that $E_i$ is $T_i$-invariant and that
\[|\sfE(A\,|\,\pi_0) - \mu(A\,|\,E_1\cap E_2)1_{E_1\cap E_2}|\cdot 1_{E_1}\cdot 1_{E_2} = |\sfE(A\,|\,\pi_0) - \mu(A\,|\,E_1\cap E_2)1_{E_1\cap E_2}|,\]
which in turn holds because $E_1\cap E_2$ is $\pi_0$-measurable while $A \subset E_1\cap E_2$, so that both $\sfE(A\,|\,\pi_0)$ and $1_{E_1\cap E_2}$ are still supported on $E_1\cap E_2$.

This integral (which no longer involves the trilinear form $\l$) may now be identified as
\begin{multline*}
\mu(E_1\cap E_2)^2\|\sfE(A\,|\,\pi_0) - \mu(A\,|\,E_1\cap E_2)1_{E_1\cap E_2}\|_{L^1(\mu(\cdot\,|\,E_1\cap E_2))}\\ \leq \mu(E_1\cap E_2)^2\|\sfE(A\,|\,\pi_0) - \mu(A\,|\,E_1\cap E_2)1_{E_1\cap E_2}\|_{L^2(\mu(\cdot\,|\,E_1\cap E_2))},
\end{multline*}
using the fact that $\S^{(1,0)}$ and $\S^{(0,1)}$ are independent to write $\mu(E_1)\mu(E_2) = \mu(E_1\cap E_2)$ and using H\"older's inequality for the final upper bound.

\emph{Second term}\quad This is much simpler: since $A \subset E_1\cap E_2$ and $E_i$ is $T_i$-invariant we have
\begin{eqnarray*}
\l(E_1\cap E_2,A,A) &=& \lim_{N\to\infty}\frac{1}{N}\sum_{n=1}^N\mu((E_1\cap E_2)\cap T_1^{-n}A\cap T_2^{-n}A)\\
&=& \lim_{N\to\infty}\frac{1}{N}\sum_{n=1}^N\mu(T_1^{-n}E_1\cap T_1^{-n}A\cap T_2^{-n}E_2\cap T_2^{-n}A)\\
&=& \lim_{N\to\infty}\frac{1}{N}\sum_{n=1}^N\mu(T_1^{-n}A\cap T_2^{-n}A)\\
&=& \lim_{N\to\infty}\frac{1}{N}\sum_{n=1}^N\mu(A\cap (T_2T_1^{-1})^{-n}A)\\
&=& \|\sfE_\mu(A\,|\,\zeta_0^{(1,-1)})\|_2^2
\end{eqnarray*}
(using the Mean Ergodic Theorem for the last equality), and by another appeal to H\"older's inequality this is bounded below by
\[\|\sfE_\mu(A\,|\,\zeta_0^{(1,-1)})\|_1^2 = \mu(A)^2 = \mu(A\,|\,E_1\cap E_2)^2\mu(E_1\cap E_2)^2.\]

\emph{Combining the estimates}\quad Using the inequalities just obtained in our original decomposition of $\l(A,A,A)$ we find that
\begin{multline*}
0 = \l(A,A,A) \geq \mu(A\,|\,E_1\cap E_2)^3\mu(E_1\cap E_2)^2\\
 - \|\sfE(A\,|\,\pi) - \mu(A\,|\,E_1\cap E_2)\|_{L^2(\mu(\cdot\,|\,E_1\cap E_2))}\cdot \mu(E_1\cap E_2)^2,
\end{multline*}
so re-arranging gives the desired result.
\end{proof}

\subsection{Shkredov's version of the density increment}

We can now present Shkredov's main increment result (which corresponds roughly to the conjunction of Proposition~\ref{prop:erg-dens-inc} and Corollary~\ref{cor:erg-dens-inc} in the one-dimensional setting):

\begin{prop}\label{prop:Shk-inc}
There is a nondecreasing function $c:(0,1] \to (0,1]$ which is bounded away from $0$ on compact subsets of $(0,1]$ and has the
following property.  If $(X \supset E_1\cap E_2\supset A,\mu,T_1,T_1)$ is such that
\begin{itemize}
\item[(i)] $\mu(A) > 0$,
\item[(ii)] the return-set of $A$ contains no nontrivial corners, and
\item[(iii)] $E_1\perp T_2^{-n}(E_1)$ for all $n\neq 0$ and similarly for $E_2$,
\end{itemize}
and if we set $\delta:= \mu(A\,|\,E_1\cap E_2)$, then there exists
another augmented process
\[(X'\supset E'_1\cap E'_2\supset A',\mu',T'_1,T'_2)\]
having the analogous properties (i-iii) and such that
\[\mu'(A'\,|\,E'_1 \cap E'_2) \geq \delta + c(\delta).\]
\end{prop}

\noindent\emph{Remark.}\quad Shkredov's argument
does \emph{not} give any effective control over the size of the sets
$E'_i$ in terms of the $E_i$ --- in particular, it could happen that they are very much smaller --- but the point is that this is not needed. \fin

\begin{proof}
This breaks naturally into two steps.

\emph{Step 1}\quad Extending $(X,\mu,T_1,T_2)$ and lifting $A$ and the $E_i$ if necessary, we may assume the system is pleasant. Now by Proposition~\ref{prop:Shk-main-estimate} conditions (i) and (ii) imply that
\[\|\sfE(A\,|\,\pi_0) - \mu(A\,|\,E_1\cap E_2)\|_{L^2(\mu(\cdot\,|\,E_1\cap E_2))} \geq \delta^3,\]
and hence there is some non-negligible $\pi_0$-measurable set $F$ such that
\[\mu(A\,|\,F) > \delta + \delta^3/2.\]
Moreover, since $\pi_0$ is generated by $\zeta_0^{(1,0)}$ and $\zeta_0^{(0,1)}$, after approximating this $F$ by a disjoint union of intersections of $T_1$- or $T_2$-invariant sets we may assume that it is itself of the form $F_1\cap F_2$ for some $F_1 \in \zeta_0^{(1,0)}$, $F_2 \in \zeta_0^{(0,1)}$.

Naively we should like to replace $A \subset E_1\cap E_2$ with $A\cap F_1\cap F_2 \subset F_1\cap F_2$, but these sets $F_i$ may not satisfy $F_i\perp T^n F_i$ for $n\neq 0$.  We resolve this by another conditioning and a vague limit construction.  Note at this point that this selection of the sets $F_i$ will be responsible for our lack of control over $\mu'(E_i')$ in terms of $\mu(E_i)$.

\emph{Step 2}\quad Let \[X' := (\{0,1\}\times \{0,1\}\times \{0,1\})^{\bbZ^2}\]
with its product Borel space structure, let $T_1', T_2'$ be the two coordinate-shifts on this space, and let $E_1'$, $E_2'$ and $A'$ be the three obvious time-zero cylinder sets of $X'$:
\[E_1' := \{(\o^1_\bf{n},\o^2_\bf{n},\o^3_\bf{n})\in X':\ \o^1_{\bs{0}} = 1\}\quad\quad \hbox{and similarly.}\]

We will show that for any $\eps > 0$ and $K\geq 1$ there is a probability measure $\nu$ on $X'$ such that
\begin{itemize}
\item $\nu$ is approximately invariant: $|\nu(C) - \nu((T_i')^{-1}C)| < \eps$ for any $C \subset X'$ and $i=1,2$,
\item $\nu(E'_1\cap E'_2) \geq (\delta^3/20)\mu(F_1\cap F_2)$,
\item $\nu(E'_i\triangle T'_iE'_i) = 0$ for $i=1,2$,
\item $|\nu(E'_i\triangle (T'_j)^{-k}E'_i) - \nu(E_i')^2| < \eps$ for all nonzero $-K \leq k \leq K$ for $\{i,j\} = \{1,2\}$, and
\item $\nu(A'\,|\,E_i'\cap E_2') \geq \delta + \delta^3/2$.
\end{itemize}

Given this, we may take a sequence of such measures as $\eps \downarrow 0$ and $K\to\infty$ and let $\mu'$ be a vague limit of some subsequence to obtain an augmented process
\[(X'\supset E'_1\cap E'_2\supset A',\mu',T_1',T_2')\]
having all the desired properties. The $T_i'$-invariance of $\mu'$ follows from the approximate invariance of the measures $\nu$, and the $T_i'$-invariance of $E'_i$ is only up to a $\mu'$-negligible set, but this may then be repaired by replacing $E'_i$ with $\bigcup_n (T_i')^nE_i'$, which differs from $E_i'$ only by a $\mu'$-negligible set.  The second of the above points ensures that the limit $\mu'$ is non-trivial insofar as $\mu'(A'),\mu'(E_1'\cap E_2') > 0$.

Now fix $\eps$ and $K$.  To obtain such a $\nu$, let $(Z,m,R_1,R_2)$ be a compact group rotation isomorphic to the Kronecker factor of $(X,\mu,T_1,T_2)$ with factor map $\zeta_1^T =: \zeta:X\to Z$, and let $\cal{U}$ be a Borel partition of $Z$ into sufficiently small pieces that
\[\big\|\sfE_\mu(F_i\,|\,\zeta)|_U - \mu(F_i\,|\,U)\big\|_{L^2(\mu(\cdot\,|\,U))} < \eps/4\]
for all $U \in \cal{U}\setminus \cal{U}_{\rm{bad}}$ where $m\big(\bigcup\cal{U}_{\rm{bad}}\big) < (\delta^3/20) \mu(F_1\cap F_2)$.

Considering the convex combination
\[\mu(A\,|\,F_1\cap F_2) = \sum_{U \in \cal{U}}\frac{\mu(F_1\cap F_2\cap \zeta^{-1}U)}{\mu(F_1\cap F_2)}\mu(A\cap \zeta^{-1}U\,|\,F_1\cap F_2\cap \zeta^{-1}U),\]
the terms indexed by $\cal{U}_{\rm{bad}}$ must contribute very little (because their sum cannot be more than $\delta^3/20$ if we estimate by simply ignoring the factors of $\mu(A\cap \zeta^{-1}U\,|\,F_1\cap F_2\cap \zeta^{-1}U) \leq 1$). Similarly, the terms for which
\[\mu(F_1\cap F_2\cap \zeta^{-1}U\,|\,F_1\cap F_2) < (\delta^3/20)m(U)\]
must also contribute very little (their sum is also less than $\delta^3/20$).  Therefore there must be some $U \in \cal{U}\setminus \cal{U}_{\rm{bad}}$ for which 
\[\mu(F_1\cap F_2\cap \zeta^{-1}U\,|\,F_1\cap F_2) \geq (\delta^3/20)m(U)\]
and
\[\mu(A\cap \zeta^{-1}U\,|\,F_1\cap F_2\cap \zeta^{-1}U) \geq \delta + \delta^3/4.\] Using Bayes' formula, the first of these inequalities implies that
\begin{eqnarray*}
\mu(F_1\cap F_2\,|\,\zeta^{-1}U) &=& \mu(F_1\cap F_2\cap \zeta^{-1}U\,|\,F_1\cap F_2)\cdot \frac{\mu(F_1\cap F_2)}{m(U)}\\
&\geq& (\delta^3/20)\mu(F_1\cap F_2).
\end{eqnarray*}

Now let $V \subset \bbZ$ be the Bohr set
\[\{n \in \bbZ:\ m(U\triangle R_1^nU) < \eps m(U)/2\ \hbox{and}\ m(U\triangle R_2^nU) < \eps m(U)/2\}.\]
This is nontrivial because the rotation orbit $z\mapsto 1_{z+U}$ is continuous from $Z$ to $L^2(m)$, and so $V$ has some (perhaps very small) positive density in $\bbZ$.  In view of this positive density, Lemma~\ref{lem:weak-rel-weak-mix} implies that each of the sets
\begin{eqnarray*}
&&V_{j,k} := \\
&&\Big\{n \in V:\ \Big|\mu(F_i\cap T_j^{-kn}F_i\,|\,\zeta^{-1}U)\\
&&\quad\quad\quad\quad\quad\quad\quad\quad\quad - \frac{1}{m(U)}\int_U\sfE_\mu(F_i\,|\,\zeta)\cdot\sfE_\mu(T_j^{-kn}F_i\,|\,\zeta)\,\d m\Big| \leq \eps/2\Big\}
\end{eqnarray*}
still has relative density $1$ inside $V$ for any $k\neq 0$ and $j=1$ or $2$, because the whole set $\bbZ\setminus V_{j,k}$ has density zero.  Hence we may choose some $r \in V$, $r\geq 1$ that lies in every $V_{j,k}$ for $j=1,2$ and $k \in \{-K,-K+1,\ldots,K\}\setminus\{0\}$. On the other hand, the approximation that defines the members of $\cal{U}\setminus \cal{U}_{\rm{bad}}$ and the approximate return of $U$ to itself under $R_j^n$ for $n \in V$ imply that
\begin{multline*}
\frac{1}{m(U)}\int_U\sfE_\mu(F_i\,|\,\zeta)\cdot\sfE_\mu(T_j^{-kr}F_i\,|\,\zeta)\,\d m\\ \approx \frac{1}{m(U)}\int_U\mu(F_i\,|\,\zeta^{-1}U)\cdot\sfE_\mu(T_j^{-kr}F_i\,|\,\zeta)\,\d m\\
= \mu(F_i\,|\,\zeta^{-1}U)\cdot \mu(F_i\,|\,T_j^{kr}\zeta^{-1}U) \approx \mu(F_i\,|\,\zeta^{-1}U)^2
\end{multline*}
for all nonzero $-K \leq k \leq K$, where the error incurred is at most $\eps/4 + \eps/4 = \eps/2$.

Now consider the map
\[\varphi:X\to X': x\mapsto \big(1_{F_1}(T_2^{rn_2}x),1_{F_2}(T_1^{rn_1}x),1_{A\cap F_1\cap F_2}(T_1^{rn_1}T_2^{rn_2}x)\big)_{(n_1,n_2) \in \bbZ^2}\]
and let $\nu$ be the image measure $\varphi_\#\mu(\,\cdot\,|\,\zeta^{-1}U)$ on $X'$.  We will show that this has the five desired properties:
\begin{itemize}
\item approximate invariance of $\nu$ follows from approximate invariance of $U$ along $V$:
\begin{eqnarray*}
|\nu(C) - \nu((T_i')^{-1}C)| &=&  |\mu(\varphi^{-1}C\,|\,\zeta^{-1}U) - \mu(T_i^{-r}\varphi^{-1}C\,|\,\zeta^{-1}U)|\\
&=& \frac{\mu(\varphi^{-1}C\cap (\zeta^{-1}U\triangle T_i^r\zeta^{-1}U))}{\mu(\zeta^{-1}U)}\\
&\leq& \eps/2 < \eps
\end{eqnarray*}
for any $C \subset X'$;
\item a simple calculation gives
\[\nu(E_1'\cap E_2') = \mu(F_1\cap F_2\,|\,\zeta^{-1}U),\]
and this is at least $(\delta^3/20)\mu(F_1\cap F_2)$ by our choice of $U$;
\item similarly,
\[\nu(E'_i\triangle T_i'E_i') = \mu(F_i\triangle T_iF_i\,|\,\zeta^{-1}U) = 0\]
for $i=1,2$;
\item for any nonzero $-K \leq k \leq K$ we have
\[\nu(E_i'\cap (T_j')^{-k}E'_i) = \mu(F_i\cap T_j^{-kr}F_i\,|\,\zeta^{-1}U),\]
and by our selection of $r$ this is within $\eps/2$ of
\[\frac{1}{m(U)}\int_U\sfE_\mu(F_i\,|\,\zeta)\cdot\sfE_\mu(T_j^{-kn}F_i\,|\,\zeta)\,\d m,\]
which in turn is within $\eps/2$ of
\[\mu(F_i\,|\,\zeta^{-1}U)^2 = \nu(E_i')^2,\]
giving the required estimate;
\item lastly, our choice of $U$ also guarantees that
\[\nu(A'\,|\,E_1'\cap E_2') = \mu(A\cap \zeta^{-1}U\,|\,F_1\cap F_2\cap \zeta^{-1}U) \geq \delta + \delta^3/4,\]
as required.
\end{itemize}
This completes the proof with $c(\delta) := \delta^3/4$.
\end{proof}

\begin{rmk}
The two steps above can also be loosely identified with two steps in
Shkredov's work. The first is similar to the conjunction of Lemma 11
and Proposition 3 in Section 3 of~\cite{Shk06}, whose use appears at
the beginning of the proof of Theorem 4.  The second, rather more
involved, amounts to Corollary 1 and the various auxiliary results
needed to reach it in Section 4 of~\cite{Shk06}, which then underpin
the second step of each increment in the proof of Shkredov's Theorem
4. \fin
\end{rmk}

\begin{proof}[Proof of Theorem~\ref{thm:Shk}]
This now proceeds almost exactly as for Theorem~\ref{thm:MRT}.

Suppose there exists an augmented process $(X\supset E_1\cap E_2 \supset A,\mu,T_1,T_2)$ such that $\mu(A) > 0$ and hence $\mu(A\,|\,E_1\cap E_2) =: \delta_0 > 0$, $E_i \perp T_j^n(E_i)$ for all $n\geq 0$, and for which
\[\frac{1}{N}\sum_{n=1}^N\mu(A\cap T_1^{-n}A\cap T_2^{-n}A) \to 0.\]
In particular, if $(X\supset A,\mu,T_1,T_2)$ is a process violating Theorem~\ref{thm:Shk}, then $(X\supset X\cap X\supset A,\mu,T_1,T_2)$ is an augmented process with these properties.

From these data one can construct another augmented process $(Y\supset G_1\cap G_2\supset B,\nu,S_1,S_2)$ such that $\nu(B) = \mu(A)$, $\nu(G_i) = \mu(E_i)$ and this new process actually has no corners in its return set.  This construction proceeds in exact analogy with Step 1 in the proof Theorem~\ref{thm:multiMRT} from Corollary~\ref{cor:erg-dens-inc}: the initial process is transferred to the symbolic space
\[Y := (\{0,1\}\times \{0,1\}\times \{0,1\})^{\bbZ^2},\]
where now the three copies of $\{0,1\}$ above the coordinate $(n_1,n_2)$ receive the indicator functions of $1_{E_1}\circ T_1^{n_1}T_2^{n_2}$, $1_{E_2}\circ T_1^{n_1}T_2^{n_2}$ and $1_A\circ T_1^{n_1}T_2^{n_2}$ respectively; and then averaging over dilations constructs a new shift-invariant measure on this symbolic space that retains the properties of the original system but actually has no corners in its return set.  A quick check shows that if $G_1$, $G_2$ and $B$ denote the one-dimensional cylinder sets defined by the three different $\{0,1\}$-valued coordinates above $(0,0)$ in $Y$, then the $G_i$ retain the property of $\nu$-a.s. invariance under $S_i$ and also the property that $G_i\perp S_j^n(G_i)$ for all $n\neq 0$ (because the measure $\nu(G_i\cap S_j^n(G_i))$ is obtained as an average over $m$ of $\mu(E_i\cap T_j^{nm}(E_i))$, and these are all equal to $\mu(E_i)^2 = \nu(G_i)^2$ by assumption).

Now implementing Proposition~\ref{prop:Shk-inc}, one can construct from $(Y\supset G_1\cap G_2\supset B,\nu,S_1,S_2)$ a new augmented process $(X' \supset E_1'\cap E'_2 \supset A',\mu',T'_1,T'_2)$ which still has all the properties (i--iii) and for which $\mu(A'\,|\,E_1'\cap E'_2)\geq \delta_0 + c(\delta_0)$.  Since $c$ is uniformly positive on $[\delta_0,1]$, after iterating this construction finitely many times we obtain an example of an augmented process for which this relative density is greater than $1$, a contradiction.
\end{proof}

\begin{rmk}
The above treatment bears comparison with how Shkredov assembles the various components of the proof of his main result, Theorem 4, in~\cite{Shk06}. \fin
\end{rmk}

\section{Further discussion}\label{sec:2Dprobs}

Theorem~\ref{thm:Shk} remains the most elaborate higher-dimensional case of Theorem~\ref{thm:multiSzem} to be successfully proved using a density-increment argument, or to be given bounds that improve over the hypergraph-regularity proofs of the general theorem obtained in~\cite{Gow07} and~\cite{NagRodSch06}.  Perhaps the most obvious obstruction to further progress is that the various `inverse theorems' that are known for the relevant notions of uniformity remain incomplete.  However, in the ergodic-theoretic world these correspond to `characteristic factor' theorems such as Theorem~\ref{thm:Fberg-struct}, and recent work has in fact taken these a little further.  The following result appears (in a slightly more general form) as Theorem 1.1 in~\cite{Aus--lindeppleasant2}, where it is used for a different purpose.

\begin{thm}\label{thm:super-Fberg-struct}
Any ergodic $\bbZ^2$-system $(X^\circ,\mu^\circ,T_1^\circ,T_2^\circ)$ admits an ergodic extension
\[\pi:(X,\mu,T_1,T_2)\to (X^\circ,\mu^\circ,T^\circ_1,T_2^\circ)\]
with the property that
\begin{multline*}
\frac{1}{N}\sum_{n=1}^N\int_X f_0\cdot (f_1\circ T_1^n)\cdot (f_2\circ T_2^n)\cdot (f_3\circ T_1^nT_2^n)\,\d\mu\\
\sim \frac{1}{N}\sum_{n=1}^N\int_X \sfE(f_0\,|\,\pi_0)\cdot (\sfE(f_1\,|\,\pi_1)\circ T_1^n)\cdot (\sfE(f_2\,|\,\pi_2)\circ T_2^n)\cdot (\sfE(f_3\,|\,\pi_3)\circ T_1^nT_2^n)\,\d\mu
\end{multline*}
in $L^2(\mu)$ as $N\to\infty$ for any $f_0,f_1,f_2,f_3 \in L^\infty(\mu)$, where
\begin{eqnarray*}
\pi_0 = \pi_3 &:=& \zeta_0^{(1,0)}\vee \zeta_0^{(0,1)}\vee \zeta_0^{(1,1)}\vee \zeta_{2,\nil}^T\\
\pi_1 = \pi_2 &:=& \zeta_0^{(1,0)}\vee \zeta_0^{(1,-1)}\vee \zeta_0^{(0,1)}\vee \zeta_{2,\nil}^T,
\end{eqnarray*}
and where $\zeta_{2,\nil}^T$ denotes a factor generated by an inverse limit of a sequence of actions of $\bbZ^2$ by two-step nilrotations.
\end{thm}

Once again, these $\pi_i$ are referred to as the `characteristic' factors for these multiple averages.

Moreover, a relatively simple extension of Lemma~\ref{lem:joint-dist-of-part-invts} shows that the four factors $\zeta_0^{(1,0)}$, $\zeta_0^{(0,1)}$, $\zeta_0^{(1,1)}$ and $\zeta_0^{(1,-1)}$ that appear above are relatively independent over their further intersections with $\zeta_{2,\nil}^T$ (see Proposition 5.3 in~\cite{Aus--lindeppleasant2}).  Theorem~\ref{thm:super-Fberg-struct} and this second result are both known special cases of a general conjecture on the joint distributions of partially invariant factors of $\bbZ^d$-systems, which may be found formulated carefully in Section 6 of~\cite{Aus--thesis} and which suggests that an inverse theory for all higher-dimensional notions of uniformity generalizing the Gowers norms will ultimately be available.

Theorem~\ref{thm:super-Fberg-struct} itself bears on the special case of multiple recurrence asserting that 
\[\mu(A) > 0\quad\Rightarrow \quad \lim_{N\to\infty}\frac{1}{N}\sum_{n=1}^N\mu(A\cap T_1^{-n}A\cap T_2^{-n}A\cap T_1^{-n}T_2^{-n}A) > 0,\]
which in the finitary world corresponds to finding squares in dense subsets of $\bbZ^2$.  The above structural results offer hope that some analog of Shkredov's density-increment approach may be possible through the study of pleasant augmented processes of the form
\[(X\supset E_1\cap E_2\cap E_3\cap E_4 \supset A,\mu,T_1,T_2)\]
where $E_1$, $E_2$, $E_3$ and $E_4$ are measurable with respect to $\zeta_0^{(1,0)}$, $\zeta_0^{(0,1)}$, $\zeta_0^{(1,1)}$ and $\zeta_0^{(1,-1)}$ respectively. Of course, more ideas would still be needed to give a new density-increment proof of this instance of multiple recurrence, even in the infinitary setting of ergodic theory.  For example, Proposition~\ref{prop:Shk-main-estimate} must be replaced with some more complicated estimate, and then arguments in the previous section which used some conditioning on the Kronecker factor would presumably be replaced by conditioning on $\zeta_{2,\nil}^T$, which can have much more complicated behaviour.

Another interesting issue on which ergodic theory can shed some light concerns the difference between the problems of proving multiple recurrence for the above averages and for the averages
\[\frac{1}{N}\sum_{n=1}^N\mu(A\cap T_1^{-n}A\cap T_2^{-n}A\cap T_3^{-n}A)\]
arising from a $\bbZ^3$-system $(X,\mu,T_1,T_2,T_2)$.  We offer only a very informal discussion of this here, since precise results on these more complex problems are still in their infancy.  In the finitary world, these latter averages correspond to finding three-dimensional corners in dense subsets of $\bbZ^3$, rather than squares in $\bbZ^2$.  Since a triple of the form $(T_1,T_2,T_1T_2)$ does formally generate an action of $\bbZ^3$, it is clear that multiple recurrence for these $\bbZ^3$-system averages is at least as strong as its counterpart for the averages of Theorem~\ref{thm:super-Fberg-struct}.  However, the identification of characteristic factors for the case of $\bbZ^3$-systems is also apparently simpler: the main result of~\cite{Aus--newmultiSzem} shows that, after passing to a suitable extension if necessary, one has
\begin{multline*}
\frac{1}{N}\sum_{n=1}^N\int_X f_0\cdot (f_1\circ T_1^n)\cdot (f_2\circ T_2^n)\cdot (f_3\circ T_3^n)\,\d\mu\\
\sim \frac{1}{N}\sum_{n=1}^N\int_X \sfE(f_0\,|\,\pi'_0)\cdot (\sfE(f_1\,|\,\pi'_1)\circ T_1^n)\cdot (\sfE(f_2\,|\,\pi'_2)\circ T_2^n)\cdot (\sfE(f_3\,|\,\pi'_3)\circ T_3^n)\,\d\mu
\end{multline*}
with
\[\pi'_0 := \zeta_0^{(1,0,0)}\vee \zeta_0^{(0,1,0)}\vee \zeta_0^{(0,0,1)}, \quad\quad\quad \pi'_1 := \zeta_0^{(1,0,0)}\vee \zeta_0^{(1,-1,0)}\vee \zeta_0^{(1,0,-1)},\]
\[\pi'_2 := \zeta_0^{(1,-1,0)}\vee \zeta_0^{(0,1,0)}\vee \zeta_0^{(0,1,-1)} \quad\hbox{and}\quad \pi'_3 := \zeta_0^{(1,0,-1)}\vee \zeta_0^{(0,1,-1)}\vee \zeta_0^{(0,0,1)},\]
and these are the minimal factors with this property. These factors are `simpler' in that they involve only partially-invariant factors, and not compact group rotations or nilsystems.  The fact that some of the ingredients needed in Theorem~\ref{thm:super-Fberg-struct} no longer appear here does not contradict the fact that a triple such as $(T_1,T_2,T_1T_2)$ generates a $\bbZ^3$-system, because after passing to a suitable extension the algebraic relations among the generators of this $\bbZ^3$-system will usually be lost.

It thus appears that the analysis of the more general averages might actually be easier, and in fact for deploying some of the methods at our disposal this is true.  The ergodic theoretic proof of convergence of these averages in~\cite{Aus--nonconv} (reproving a result of Tao from~\cite{Tao08(nonconv)}) implicitly needs the linear independence of the group elements corresponding to $T_1$, $T_2$ and $T_3$. In the finitary world, the hypergraph-regularity proofs of the multidimensional Szemer\'edi Theorem must first lift the problem into a group $\bbZ^d$ for $d$ large enough so that one is looking for the corners of a $d$-dimensional simplex (rather than any more complicated $d$-dimensional constellations) before this search can be correctly recast in the language of extremal hypergraph theory.

However, both of these arguments use only the most basic, `rough' structure for the data being studied, and by contrast the more refined density-increment approach \emph{is} simpler in the case of two-dimensional squares than that of three-dimensional corners.  Each of the superficially-simpler characteristic factors $\pi_i'$ for the three-dimensional problem is assembled from ingredients of the form $\zeta_0^{\bf{v}}$ for some $\bf{v} \in \bbZ^3$, and each of these is a factor map onto a factor of $(X,\mu,T_1,T_2,T_3)$ on which the acting group is essentially $\bbZ^3/\bbZ\bf{v}\cong \bbZ^2$ (owing to the partial invariance).  In order to mimic Shkredov's approach to these results, it is then necessary to know how all of these essentially two-dimensional systems are jointly distributed as factors of $(X,\mu,T_1,T_2,T_3)$ (in order to generalize our use of Lemma~\ref{lem:joint-dist-of-part-invts} in the proof of Proposition~\ref{prop:Shk-main-estimate}, for example).  It turns out that to understand this joint distribution one needs the same kind of machinery as for the identification of a tuple of characteristic factors in the first place (the reason why these are essentially equivalent problems is discussed in detail in Chapter 4 of~\cite{Aus--thesis}); and the particular problem of describing the joint distribution of these `two-dimensional' factors turns out to be of a similar level of complexity to the problem of describing characteristic factors for multiple recurrence across squares in a $\bbZ^2$-action.

So the finer information required for the density-increment strategy forces one to understand not only the `top-level' structural result that is contained in the identification of a characteristic tuple of factors, but also how all the ingredients appearing in those characteristic factors are jointly distributed.  This can be of similar difficulty to a lower-dimensional problem of identifying characteristic factors.  For understanding multiple recurrence across translates and dilates of some complicated constellation in $\bbZ^d$, one might need to work with a large partially ordered family of factors of a given system, where the characteristic factors appear as the maximal elements, and several layers of smaller factors (including group rotations, nilsystems, or possibly something else) must also be identified in order to describe all the necessary joint distributions well enough to implement a density increment.  For a density-increment proof such as in Section~\ref{sec:2D} above, this would presumably entail working with a much richer analog of the augmented processes that appear there.

These speculations notwithstanding, serious problems surround the status of finitary analogs of Theorem~\ref{thm:super-Fberg-struct} or its generalizations.  I believe such analogs are expected by many researchers in this field, but formulating a precise conjecture is already tricky, and at this writing I know of no higher-dimensional results beyond Shkredov's.  It is not clear what methods (extending Shkredov's or others) are needed to establish such structural results.  Without them, the prospect of a density-increment proof of the presence of prescribed constellations in dense subsets of $\bbZ^d$ seems rather remote.

\bibliographystyle{alpha}
\bibliography{bibfile}

\vspace{20pt}

\parskip 0pt

\end{document}